\documentclass[12pt]{amsart}
\usepackage{amssymb,times,epsfig}

\makeatletter
\def\@strippedMR{}
\def\@scanforMR#1#2#3\endscan{
  \ifx#1M\ifx#2R\def\@strippedMR{#3}
  \else\def\@strippedMR{#1#2#3}
  \fi\fi}
\renewcommand\MR[1]{\relax\ifhmode\unskip\spacefactor3000 \space\fi
  \@scanforMR#1\endscan
  MR\MRhref{\@strippedMR}{\@strippedMR}}
\makeatother

\parskip=\medskipamount

\newtheorem*{Thm*}{Theorem}
\newtheorem{Thm}{Theorem}
\newtheorem{Cor}[Thm]{Corollary}
\newtheorem{Prop}[Thm]{Proposition}
\newtheorem{Lemma}[Thm]{Lemma}

\theoremstyle{definition}
\newtheorem{Defn}{Definition}

\newtheorem{Remark}{Remark}
\newtheorem{Example}{Example}

\newcommand{\mf}[1]{\mathbb{#1}}
\newcommand{\mc}[1]{\mathcal{#1}}

\DeclareMathOperator{\NC}{\mathrm{NC}}
\DeclareMathOperator{\Part}{\mathcal{P}}
\DeclareMathOperator{\Int}{\mathit{Int}}
\DeclareMathOperator{\Inner}{\mathrm{Inn}}
\DeclareMathOperator{\Outer}{\mathrm{Out}}

\DeclareMathOperator{\Var}{\mathrm{Var}}
\DeclareMathOperator{\supp}{\mathrm{supp}}

\newcommand{\abs}[1]{\left\vert#1\right\vert}

\newcommand{\set}[1]{\left\{#1\right\}}

\renewcommand{\phi}{\varphi}

\allowdisplaybreaks[1]

\title{Semigroups of distributions with linear Jacobi parameters}
\author[Anshelevich \& M{\l}otkowski]{Michael Anshelevich and Wojciech M{\l}otkowski}
\thanks{M.A. was supported in part by NSF grant DMS-0900935.
W.M. was supported by MNiSW: N~N201~364436.}
\address{Department of Mathematics, Texas A\&M University, College Station, TX 77843-3368}
\email{manshel@math.tamu.edu}
\address{Mathematical Institute,
University of Wroc\l aw,
Pl.~Grunwaldzki~2/4,
50-384~Wroc\l aw, Poland}
\email{mlotkow@math.uni.wroc.pl}
\date{\today}

\begin{document}

\begin{abstract}
We show that a convolution semigroup $\{\mu_t\}$ of measures has Jacobi parameters polynomial in the convolution parameter $t$ if and only if the measures come from the Meixner class. Moreover, we prove the parallel result, in a more explicit way, for the free convolution and the free Meixner class. We then construct the class of measures satisfying the same property for the two-state free convolution. This class of two-state free convolution semigroups has not been considered explicitly before. We show that it also has Meixner-type properties. Specifically, it contains the analogs of the normal, Poisson, and binomial distributions, has a Laha-Lukacs-type characterization, and is related to the $q=0$ case of quadratic harnesses.
\end{abstract}

\maketitle

\section{Introduction.}

Any probability measure $\mu$ on the real line, all of whose moments are finite, has associated to it two sequences of Jacobi parameters $\set{\beta_i, \gamma_i}$: for example, $\mu$ is the spectral measure of the tridiagonal matrix
\[
\begin{pmatrix}
\beta_0 & \gamma_0 & 0 & 0 & \ddots \\
1 & \beta_1 & \gamma_1 & 0 & \ddots \\
0 & 1 & \beta_2 & \gamma_2 & \ddots \\
0 & 0 & 1 & \beta_3 & \ddots \\
\ddots & \ddots & \ddots & \ddots & \ddots
\end{pmatrix}.
\]
We will denote this fact by
\[
J(\mu) =
\begin{pmatrix}
\beta_0, & \beta_1, & \beta_2, & \beta_3, & \ldots \\
\gamma_0, & \gamma_1, & \gamma_2, & \gamma_3, & \ldots
\end{pmatrix},
\]
with $\beta_n(\mu):=\beta_n$, $\gamma_n(\mu):=\gamma_n$. These parameters are related to the moments of the measure via the Viennot-Flajolet \cite{Flajolet,Viennot-Notes} and Accardi-Bo\.zejko \cite{AccBozGaussianization} formulas. On the other hand, in probability theory and other applications, measures frequently come in time-dependent convolution semigroups. In general, the time dependence of the Jacobi parameters is complicated (they are rational functions of $t$). However, for the Gaussian convolution semigroup
\[
\mu_t(x) = \frac{1}{\sqrt{2 \pi t}} e^{-x^2/2t},
\]
the Jacobi parameters are simply
\[
\beta_n(t) = 0, \qquad \gamma_n(t) = (n+1) t,
 \]
while for the Poisson convolution semigroup
\[
\mu_t(x) = e^{-t} \sum_{k=0}^\infty \frac{1}{k!} t^k \delta_k(x),
\]
they are
\[
\beta_n(t) = n + t, \qquad \gamma_n(t) = (n+1) t.
\]
So it is natural to ask, what are all the measures whose Jacobi parameters are linear (in the calculus sense, that is, affine) functions of the convolution parameter? A seemingly more general question is to describe all collections $\set{\alpha_n, \beta_n, \gamma_n, a_n, b_n, c_n}$ such that the spectral measures $\set{\mu_t}$ of tridiagonal matrices
\begin{equation}
\label{Jacobi-matrix}
t \begin{pmatrix}
\beta_0 & \gamma_0 & 0 & 0 & \ddots \\
\alpha_0 & \beta_1 & \gamma_1 & 0 & \ddots \\
0 & \alpha_1 & \beta_2 & \gamma_2 & \ddots \\
0 & 0 & \alpha_2 & \beta_3 & \ddots \\
\ddots & \ddots & \ddots & \ddots & \ddots
\end{pmatrix}
+
\begin{pmatrix}
b_0 & c_0 & 0 & 0 & \ddots \\
a_0 & b_1 & c_1 & 0 & \ddots \\
0 & a_1 & b_2 & c_2 & \ddots \\
0 & 0 & a_2 & b_3 & \ddots \\
\ddots & \ddots & \ddots & \ddots & \ddots
\end{pmatrix}
\end{equation}
form a convolution semigroup. In this paper we provide the answer: such measures form precisely the Meixner class \cite{Meixner}. Thus we add a new, more dynamical description to already numerous known characterizations of this class. In particular, far from being infinite dimensional, this family of measures is described by only four parameters. For more on the Meixner class, see the probabilistic characterizations in \cite{Laha-Lukacs,Wes-commutative}, its role in statistics (where it is known under the name of quadratic exponential families) \cite{Morris,Morris-Statistical,Diaconis-Gibbs-sampling}, and a more combinatorial description \cite{KimZeng}.

The reason convolution semigroups appear in probability theory is that if $\set{X(t)}$ is a process with stationary independent increments, and $\mu_t$ is the distribution of $X(t)$, then $\set{\mu_t}$ form a convolution semigroup. In non-commutative probability theories, one encounters other notions of independence, and correspondingly other convolution operations based on them. In many ways, these operations are more complicated than the usual convolution; notably, the usual operation is distributive,
\[
\mu \ast (\nu_1 + \nu_2) = \mu \ast \nu_1 + \mu \ast \nu_2
\]
while the other ones are not. Nevertheless, in other ways they appear to be simpler. The combinatorial theory of such convolutions is typically based on an appropriate sequence of cumulants which linearize it; for example, the classical cumulants $r^\ast_n(\mu)$ defined via
\[
\sum_{n=1}^\infty \frac{(i \theta)^n}{n!} r^\ast_n(\mu) = \log \int_{\mf{R}} e^{i \theta x} \,d\mu(x)
\]
have the property that
\begin{equation}
\label{Linearizing}
r^\ast_n(\mu \ast \nu) = r^\ast_n(\mu) + r^\ast_n(\nu).
\end{equation}
In particular, $r^\ast(\mu^{\ast t}) = t \cdot  r^\ast(\mu)$: cumulants are always proportional to the convolution parameter $t$. While, as pointed out above, there is a nice relation between Jacobi parameters and moments, as well as a relation between cumulants and moments (see Section~\ref{Subsec:cumulants}), we are not aware of a simple relation between Jacobi parameters and cumulants. However, in \cite{Mlotkowski-Cumulants-Jacobi}, the second author found a formula relating Jacobi parameters and \emph{free} cumulants, which linearize the \emph{free convolution} \cite{Nica-Speicher-book}. This allows us to provide a constructive proof that Jacobi parameters are linear in the free convolution parameter if and only if the measure belongs to the free Meixner class considered in \cite{SaiConstant,AnsMeixner,Boz-Bryc} and numerous other sources. We then give a simple but indirect argument which provides the corresponding characterization for the Meixner class. Another consequence of the analysis is that there are no measures whose Jacobi parameters are polynomial functions of the free convolution parameter of degree greater than one, so the spectral measures of matrices from equation~\eqref{Jacobi-matrix} are no more general.

Yet another convolution operation was introduced in \cite{BLS96} in relation to what the authors called ``conditionally free probability'', but is better called two-state free probability theory. As the name indicates, this is a convolution operation $\boxplus_c$ on pairs of measures, and as such does not really have a classical analogue. The techniques from \cite{Mlotkowski-Cumulants-Jacobi} allow us to find all pairs of measures $(\widetilde{\mu},\mu)$ such that if $(\widetilde{\mu}_t,\mu_t):=(\widetilde{\mu},\mu)^{\boxplus_c t}$ then the Jacobi parameters of $\widetilde{\mu}_t$ are linear with respect to this convolution. In fact, it suffices to only require that the Jacobi parameters of $\widetilde{\mu}_t$ are polynomials in $t$, and the linearity of the Jacobi parameters of both $\widetilde{\mu}_t$ and $\mu_t$ then follows automatically. Unlike in the cases above, this class has not been explicitly described before. It consists of measures whose Jacobi parameters do not depend on $n$ for $n \geq 2$ (except for the special case described in Proposition~6).

Even the fact that these (pairs of) measures form a two-state free convolution semigroup is apparently new. We show that these measures also have, in the two-state context, Meixner-type properties. Namely, just like the Meixner and free Meixner classes, this family includes the two-state versions of the normal, Poisson, and binomial distributions; their two-state cumulants satisfy a quadratic recursion; they have a two-state Laha-Lukacs characterization \cite{Boz-Bryc-Two-states}; and they appear as a subclass of the $q=0$ case of quadratic harnesses \cite{Bryc-Meixner}.

\textbf{Acknowledgements.} The paper was started during the 12th workshop on Non-commutative Harmonic Analysis with Applications to Probability at the Banach center. M.A. would like to thank the organizers for an enjoyable conference. He would also like to thank W{\l}odek Bryc and Jacek Weso{\l}owski for explaining their work to him. Perhaps most importantly, we thank Laura Matusevich for pointing out a missing assumption in Theorem~\ref{Thm:Linear-two-state}, which led to Proposition~\ref{Prop:Boolean}. Finally, we are grateful to the referee for a careful reading of the paper, and numerous very useful comments and suggestions.

\section{Background.}

\subsection{Partitions.}
\label{Subsec:partitions}
A \textit{partition} of a linearly ordered set $X$
is a family $\pi$ of nonempty, pairwise disjoint subsets of $X$, called \textit{blocks} of $\pi$,
such that $\bigcup\pi=X$. A partition is \textit{noncrossing} if whenever $x_1<x_2<x_3<x_4$, $x_1,x_3\in V_1\in\pi$ and $x_2,x_4\in V_2\in\pi$
then $V_1=V_2$. Every noncrossing partition admits a natural partial order: $U\preceq V$ if there are $r,s\in V$ such that $r\le k\le s$ holds for every $k\in U.$ Now we can define {\it depth} of a block $U\in\pi$, namely
\[
d(U,\pi):=|\{V\in\pi:U\preceq V\ne U\}|.
\]
If $d(U,\pi)\ge1$ then we define the {\it derivative} of $U$ as the unique block $U'\in\pi$ such that $U\preceq U'$ and $d(U',\pi)=d(U,\pi)-1$. The derivatives of higher orders are defined by putting $V^{(k)}:=\left(V^{(k-1)}\right)'$.

In particular, a block of a noncrossing partition with $d(U, \pi) = 0$ is called \textit{outer},
and a block with $d(U, \pi) \geq 1$ is called \textit{inner}.
An \textit{interval partition} is a non-crossing partition with only outer blocks.
For the set $\{1,2,\ldots,m\}$, we will denote the lattice of all partitions by $\Part(m)$, the lattice of all noncrossing partitions by $\NC(m)$, and the lattice of all interval partitions by $\Int(m)$. In addition, $\mathrm{NC}_{1,2}(m)$ will stand for the class of all partitions $\pi\in\mathrm{NC}(m)$ such that $|V|\le 2$ holds for every $V\in\pi$. The family of all outer (resp. inner) blocks of $\pi$ will be denoted by $\mathrm{Out}(\pi)$ (resp. $\mathrm{Inn}(\pi)$).

\begin{Example}
\label{Example:Noncrossing}
The noncrossing partition $\pi=\big\{\{1,5,10\},\{2,3\},\{4\},\{6,7,9\},\{8\},\{11,12\},\{13\}\big\}$
is drawn in Figure~\ref{Figure:Noncrossing}. The blocks $\{1,5,10\}$, $\{11,12\}$, $\{13\}$ are outer (and so have depth $0$); the rest of the blocks are inner, with $d(\{6,7,9\}, \pi) = 1$ and $d(\{8\}, \pi) = 2$.
$\{8\}' = \{6,7,9\}$ and $\{2,3\}' = \{1,5,10\}$.

\begin{figure}[h,t]
\includegraphics[width=0.8\textwidth,height=0.2\textwidth]{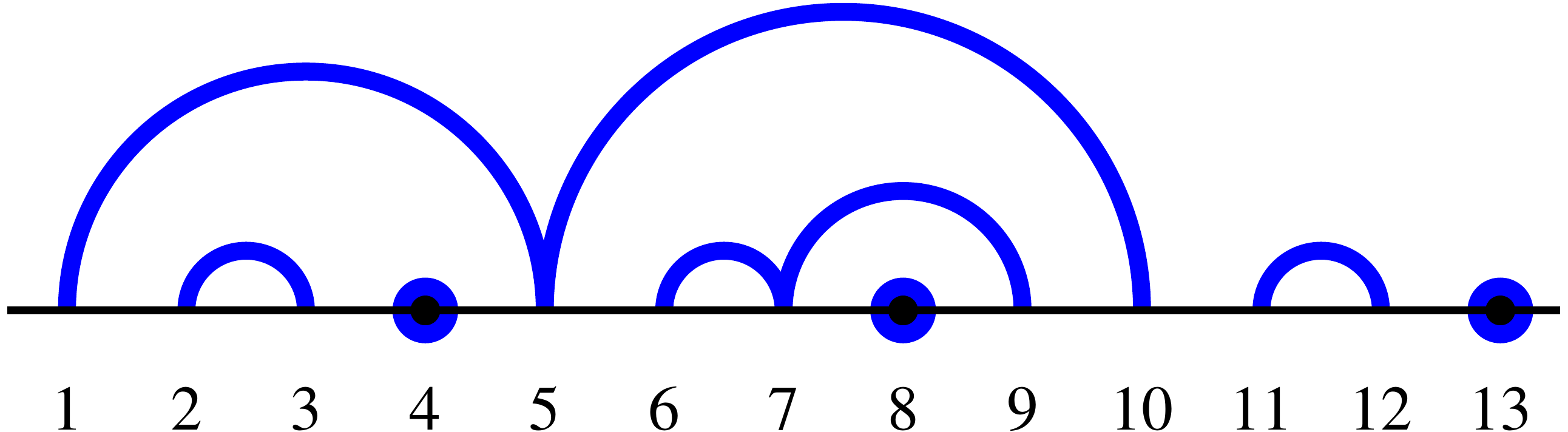}
\label{Figure:Noncrossing}
\caption{Illustration for Example~\ref{Example:Noncrossing}.}
\end{figure}
\end{Example}

\subsection{Jacobi parameters.}
\label{Subsec:Jacobi}
Throughout the paper, $\mu$ will be a probability measure on $\mathbb{R}$
all of whose \textit{moments}
\begin{equation}
s_m (\mu) :=\int_{\mathbb{R}}x^m d\mu(x)
\end{equation}
are finite.
Then there is a sequence $\{P_m\}_{m=0}^\infty$
of monic polynomials, with $\mathrm{deg}P_m=m$,
which are orthogonal with respect to $\mu$.
They satisfy a recurrence relation: $P_0(x)=1$ and for $m\ge0$
\begin{equation}
xP_m(x)=P_{m+1}(x)+\beta_m P_m(x)+\gamma_{m-1}P_{m-1}(x),
\end{equation}
under convention that $P_{-1}(x)=0$,
where the \textit{Jacobi parameters} \cite{Chihara-book} satisfy
$\beta_m\in\mathbb{R}$ and
$\gamma_m\ge0$.
Then we will write
$$
J({\mu})=\left(\begin{array}{ccccc}
\beta_0,&\beta_1,&\beta_2,&\beta_3,&\dots\\
\gamma_0,&\gamma_1,&\gamma_2,&\gamma_3,&\dots
\end{array}
\right).
$$
$\{P_m\}$ are unique for $m \leq \abs{\supp(\mu)}$. Moreover $N := \abs{\supp(\mu)} < \infty$ if and only if $\gamma_{N-1} = 0$ and $\gamma_m > 0$ for $m < N-1$. In this case for $m \geq N$, $P_{m+1}$ are not uniquely determined, and $\beta_m, \gamma_m$ are undefined, so that the sequence of Jacobi parameters is finite. By convention, we may still write infinite sequences of Jacobi parameters, but their terms after the first $\gamma_m = 0$ should be disregarded.

The Viennot-Flajolet theory \cite{Flajolet,Viennot-Notes} gives the relation between moments of a measure and its Jacobi
parameters in terms of Motzkin paths. We will use a related formula of Accardi and Bo\.zejko \cite{AccBozGaussianization} expressing the same relation using non-crossing partitions:
\begin{equation}
\label{accardiozejko}
s_m(\mu)=\sum_{\sigma\in\mathrm{NC}_{1,2}(m)}
\prod_{\substack{V\in\sigma\\|V|=1}}\beta_{d(V,\sigma)}\cdot
\prod_{\substack{V\in\sigma\\|V|=2}}\gamma_{d(V,\sigma)}.
\end{equation}
This formula should be compared with the formula~\eqref{Free-moment-cumulant} below.

\subsection{Cumulants}
\label{Subsec:cumulants}
The classical cumulants $r^\ast_n(\mu)$ \cite{Shi}, free cumulants
\[
r_n(\mu) = r_n^\boxplus(\mu)
\]
\cite{Spe90,Spe94,Nica-Speicher-book}, Boolean cumulants $r^\uplus_n(\mu)$ \cite{SW97}, and two-state free cumulants
\[
R_n(\widetilde{\mu}, \mu) = r^{\boxplus_c}_n(\widetilde{\mu}, \mu)
\]
\cite{BLS96} are defined via the following moment-cumulant formulas, which express them implicitly in terms of the moments $s_m(\mu)$:
\begin{equation}
\label{Classical-moment-cumulant}
s_m(\mu) = \sum_{\pi \in \Part(m)} \prod_{V \in \pi} r^\ast_{\abs{V}}(\mu),
\end{equation}
\begin{equation}
\label{Free-moment-cumulant}
s_m(\mu) = \sum_{\pi \in \NC(m)} \prod_{V \in \pi} r_{\abs{V}}(\mu),
\end{equation}
\begin{equation}
\label{Boolean-moment-cumulant}
s_m(\mu) = \sum_{\pi \in \Int(m)} \prod_{V \in \pi} r^\uplus_{\abs{V}}(\mu),
\end{equation}
and
\begin{equation}
\label{C-free-moment-cumulant}
s_m(\widetilde{\mu}) = \sum_{\pi \in \NC(m)} \prod_{V \in \Outer(\pi)} R_{\abs{V}}(\widetilde{\mu}, \mu) \prod_{U \in \Inner(\pi)} r_{\abs{U}}(\mu).
\end{equation}

\subsection{Convolutions}
Using cumulants, we can define in a uniform way the classical convolution $\ast$, the free convolution $\boxplus$, and the Boolean convolution $\uplus$, via equation~\eqref{Linearizing} and its analogs, for example
\begin{equation*}
r_n(\mu \boxplus \nu) = r_n(\mu) + r_n(\nu).
\end{equation*}
The two-state free (or conditionally free---these terms will be used interchangeably) convolution $\boxplus_c$  is an operation on pairs of measures, defined by
\[
R_n\bigl( (\widetilde{\mu}, \mu) \boxplus_c (\widetilde{\nu}, \nu) \bigr) = R_n(\widetilde{\mu}, \mu) + R_n(\widetilde{\nu}, \nu)
\]
and $(\widetilde{\mu}, \mu) \boxplus_c (\widetilde{\nu}, \nu) = (\widetilde{\tau}, \mu \boxplus \nu)$. Note that the classical convolution defined in this way does coincide with the more familiar formula
\[
(\mu \ast \nu)(A) = \int \mu(A - x) \,d\nu(x),
\]
but there are no such explicit formulas for the other operations. Instead, each of them is related to an appropriate notion of independence, see the references above.

For any of the convolution operations, for example for $\ast$, a \emph{convolution semigroup} generated by $\mu$ is a family of measures $\set{\mu_t}$ such that $\mu_1 = \mu$ and $\mu_t \ast \mu_s = \mu_{t + s}$. So a classical convolution semigroup is characterized by the property that
\[
r_n^\ast(\mu_t) = t \cdot r_n^\ast(\mu),
\]
and a similar relation holds between other convolution semigroups and corresponding cumulants.

A priori, our semigroups will be indexed by $t \in \mf{N}$. For $\mu$ $\ast$-infinitely divisible, $\mu^{\ast t}$ is defined for all $t \in [0, \infty)$; a similar comment applies for $\boxplus$- and $\boxplus_c$-infinitely divisible distributions. However, for free convolution $\boxplus$ and two-state free convolution $\boxplus_c$, for \emph{any} $\mu$ one can extend the semigroup to $t \in [1, \infty)$, see Lecture~14 of \cite{Nica-Speicher-book} and \cite{Belinschi-C-free-unbounded}. Moreover, for the Boolean convolution $\uplus$, any $\mu$ is infinitely divisible \cite{SW97}.

\subsection{Generating functions}
It is frequently more convenient to work with generating functions instead of moments and cumulants. For example, the Fourier transform
\[
\mc{F}_\mu(z) = \sum_{m=0}^\infty \frac{1}{m!} s_m(\mu) (i z)^m
\]
is the exponential moment generating function of $\mu$, and
\[
C_\mu(z) = \sum_{m=1}^\infty \frac{(i z)^m}{m!} r^\ast_m(\mu)
\]
is its (classical) cumulant generating function.

The ordinary moment generating function of $\mu$ is
\[
M^\mu(z) = \sum_{m=1}^\infty s_m(\mu) z^m.
\]

We denote the free cumulant generating function (also called the $R$-transform) by
\[
R^\mu(z) = \sum_{k=1}^\infty r_k(\mu) z^k
\]
and the two-state free cumulant generating function by
\[
R^{\widetilde{\mu}, \mu}(z) = \sum_{k=1}^\infty R_k(\widetilde{\mu}, \mu) z^k.
\]
Note that these are the combinatorial $R$-transforms, which differ by a factor of $z$ from the versions used in complex analysis.

We have the functional relations
\begin{equation}
\label{Relation:R-transform}
M^\mu(z) = R^\mu((1 + M^\mu(z)) z)
\end{equation}
and
\begin{equation}
\label{Relation:Two-state-transform}
\eta^{\widetilde{\mu}}(z) = (1 + M^\mu(z))^{-1} R^{\widetilde{\mu},\mu}((1 + M^\mu(z)) z).
\end{equation}
Here the eta-transform (Boolean cumulant generating function) $\eta^\mu$ satisfies
\begin{equation}
\label{Relation:Eta-transform}
\eta^\mu(z) = 1 - (1 + M^\mu(z) )^{-1}.
\end{equation}
Recall that the corresponding property in the classical case is
\[
C_\mu(z) = \log \mc{F}_\mu(z).
\]

\subsection{Continued fractions}
The Jacobi parameters of $\mu$ also appear in the continued fraction expansion
\begin{equation}
\label{Continued-fraction-M}
1 + M^\mu(z)
=
\cfrac{1}{1 - \beta_0(\mu) z -
\cfrac{\gamma_0(\mu) z^2}{1 - \beta_1(\mu) - \cfrac{\gamma_1(\mu) z^2}{\ldots}}},
\end{equation}
and from equation~\eqref{Relation:Eta-transform},
\begin{equation}
\label{Continued-fraction-Eta}
\eta^\mu(z)
= \beta_0(\mu) z +
\cfrac{\gamma_0(\mu) z^2}{1 - \beta_1(\mu) z -
\cfrac{\gamma_1(\mu) z^2}{1 - \beta_2(\mu) - \cfrac{\gamma_2(\mu) z^2}{\ldots}}}.
\end{equation}

\begin{Defn}
\emph{Meixner distributions} are measures with Jacobi parameters
\[
J(\mu) =
\begin{pmatrix}
\beta_0, & b + \beta_0, & 2 b + \beta_0, & 3 b + \beta_0, & \ldots \\
\gamma_0, & 2(c + \gamma_0), & 3 (2 c + \gamma_0), & 4 (3 c + \gamma_0), & \ldots
\end{pmatrix}
\]
for $\gamma_0 \geq 0$, and either $c \geq 0$ or $c = -\gamma_0/N$, $N \in \mf{N}$ (in the second case, the sequence of Jacobi parameters is finite). In particular this class includes the normal (Gaussian) distribution for $\beta_0 = b = c = 0$, $\gamma_0 = 1$, Poisson distribution for $\beta_0 = b = \gamma_0 = 1$, $c = 0$, binomial distributions for $\beta_0 = p N$, $\gamma_0 = p (1-p) N$, $b = 1 - 2 p$, $c = - p(1-p)$, gamma distributions for $\beta_0 = \gamma_0 = \alpha$, $b=2$, $c=1$, and negative binomial distributions for $\beta_0 = \frac{p}{1-p} r$, $\gamma_0 = \frac{p}{(1-p)^2} r$, $b = \frac{1+p}{1-p}$, $c = \frac{p}{(1-p)^2}$. See \cite{SchOrthogonal} for more details.

Moreover, for fixed $\beta_0, \gamma_0, b, c$, the measures $\set{\mu_t: t \in \mf{N}}$ with Jacobi parameters
\[
\beta_n(t) = n b + \beta_0 t, \quad \gamma_n(t) = (n+1) ( n c + \gamma_0 t )
\]
all belong to the Meixner class and form a convolution semigroup. If $c = - \gamma_0/N < 0$, the semigroup can be extended to $\set{\mu_t: t = \frac{n}{N}, n \in \mf{N}}$. $\mu$ is $\ast$-infinitely divisible if and only if $c \geq 0$, in which case the measures $\set{\mu_t: t \geq 0}$ form a convolution semigroup, a \emph{(classical) Meixner semigroup}.
\end{Defn}

\begin{Defn}
\emph{Free Meixner distribution} are measures with Jacobi parameters
\begin{equation}
\label{Free-Meixner-Jacobi}
J(\mu) =
\begin{pmatrix}
\beta_0, & b + \beta_0, & b + \beta_0, & b + \beta_0, & \ldots \\
\gamma_0, & c + \gamma_0, & c + \gamma_0, & c + \gamma_0, & \ldots
\end{pmatrix}
\end{equation}
for $\gamma_0 \geq 0$, $c + \gamma_0 \geq 0$, in other words their Jacobi parameters are independent of $n$ for $n \geq 1$. The normalized free Meixner distributions $\mu_{b,c}$ have mean $\beta_0 = 0$, variance $\gamma_0 = 1$, and parameters $b \in \mf{R}$, $c \geq -1$; general free Meixner distributions are affine transformations of these. More explicitly,
\[
d\mu_{b,c}(x) = \frac{1}{2 \pi} \cdot \frac{\sqrt{\Bigl(4 (1 + c) - (x - b)^2\Bigr)_+}}{1 + b x + c x^2} \,dx + 0, 1, \text{or } 2 \text{ atoms},
\]
see \cite{SaiConstant,AnsMeixner,Boz-Bryc}.

Free Meixner distribution with Jacobi parameters \eqref{Free-Meixner-Jacobi} is $\boxplus$-infinitely divisible if and only if $c\ge0$,
see Theorem~\ref{Thm:Linear-one-state} below. Moreover, for fixed $b, c$, these distributions form a two-parameter free convolution semigroup with respect to $\beta_0$ and $\gamma_0$. This follows from the formula for their $R$-transform in \cite{SaiConstant} or from the formula for their free cumulants in \cite{Hinz-Mlotkowski-Free-Cumulants}. In the particular case $c \geq -1$, the measure with Jacobi parameters \eqref{Free-Meixner-Jacobi} is precisely $\mu_{b,c}^{\boxplus \gamma_0} \boxplus \delta_{\beta_0}$ (note that the free convolution with a delta measure is a shift) and in this case it follows directly that
\[
\left( \mu_{b,c}^{\boxplus \gamma_0'} \boxplus \delta_{\beta_0'} \right) \boxplus \left( \mu_{b,c}^{\boxplus \gamma_0''} \boxplus \delta_{\beta_0''} \right) = \mu_{b,c}^{\boxplus (\gamma_0' + \gamma_0'')} \boxplus \delta_{\beta_0' + \beta_0''}.
\]

\end{Defn}

\begin{Remark}
\label{Remark:Free-Meixner}
There are numerous characterizations of the free Meixner class in free probability. Here is a partial list.
\begin{enumerate}
\item
The following measures all belong to the free Meixner class: free normal (semicircular) distributions have $b = c = 0$ (and so their Jacobi parameters do not depend on $n$), free Poisson (Marchenko-Pastur) distributions have $c = 0$, $b \neq 0$, and free binomial distributions correspond to $c < 0$ (including the Bernoulli distributions for $c = - \gamma_0$).
\item
The orthogonal polynomials of the measure $\mu$ have a generating function of the ``resol\-vent-type'' form $\frac{F(z)}{1 - x G(z)}$ \cite{AnsMeixner}.
\item
The free Laha-Lukacs property: two freely independent random variables $X, Y$ with the same distribution $\mu$ satisfy the property that the conditional expectation $\phi[X | X + Y]$ is linear in $X+Y$ and the conditional variance $\text{Var}[X | X + Y]$ is quadratic in $X+Y$ \cite{Boz-Bryc}.
\item
The free cumulant generating function of the measure $\mu_{b,c}$ satisfies a ``Riccati difference equation''
\begin{equation*}
\frac{R(z)}{z^2} = 1 + b\  \frac{R(z)}{z} + c \left( \frac{R(z)}{z} \right)^2,
\end{equation*}
see the single-variable case of Theorem~6 from \cite{AnsMulti-Sheffer}, or the $q=0$ case of Remark 5.4 from \cite{Boz-Bryc}.
\item
The measure $\mu$ generates a quadratic free exponential family \cite{Bryc-Cauchy-Stieltjes}.
\item
The measure $\mu$ is characterized in terms of its free Jacobi field \cite{Bozejko-Lytvynov-Meixner-I}.
\end{enumerate}
All of these properties have analogs for classical Meixner distributions, see the references in the Introduction.
\end{Remark}

\section{The free convolution.}

Formulas \eqref{accardiozejko} and \eqref{Free-moment-cumulant} relate moments of a measure to its Jacobi parameters, resp.\ free cumulants. It is also possible to find a direct relation between
free cumulants and Jacobi parameters, see \cite{Mlotkowski-Cumulants-Jacobi}.
For this purpose we will need some additional notions.

A \textit{labelling} of a noncrossing partition
$\sigma$ is a function $\kappa$ on $\sigma$
such that for any $V\in\sigma$ we have
$\kappa(V)\in\{0,1,\dots,d(V,\sigma)\}$.
For a labelling $\kappa$ of a noncrossing partition $\sigma$
we denote by $\mathcal{R}(\sigma,\kappa)$ the smallest
equivalence relation on $\sigma$ containing all the pairs
$\left(V^{(i)},V^{(j)}\right)$ with $V\in\sigma$,
$0\le i,j\le\kappa(V)$.
By $\mathrm{NCL}_{1,2}^{1}(m)$ (not to be confused with non-crossing linked partitions introduced by Dykema \cite{Dykema-Multilinear-function-series}) we will denote the family of all
pairs $(\sigma,\kappa)$ such that $\sigma\in\mathrm{NC}_{1,2}(m)$,
$\kappa$ is a labelling of $\sigma$
and $\mathcal{R}(\sigma,\kappa)=\sigma\times\sigma$.
In particular, $\sigma$ has only one outer block.

\begin{Example}
\label{Example:Connected}
In Figure~2 are drawn three labellings of the partition
$\sigma = \big\{\{1,6\},\{2,5\},\{3,4\}\big\}$: $\kappa_1(\{3,4\}) = 2$, $\kappa_2(\{3,4\}) = 1$,
$\kappa_2(\{2,5\}) = 1$, $\kappa_3(\{3,4\}) = 1$, with the rest of the values zero.
For each label $\kappa(V)$, we connect $V$ to its derivatives of order $1, 2, \ldots, \kappa(V)$.
Pictorially, $\mathcal{R}(\sigma,\kappa) = \sigma \times \sigma$ if all the blocks of $\sigma$
are connected in this fashion. $\kappa_1$ and $\kappa_2$ produce connected partitions,
while under $\mathcal{R}(\sigma,\kappa_3)$, $\{2,5\} \sim \{3,4\}$ but $\{1,6\} \not \sim \{2,5\}$.
\end{Example}

\begin{figure}[h,t]
\label{Figure:Connected}
\includegraphics[width=0.3\textwidth,height=0.15\textwidth]{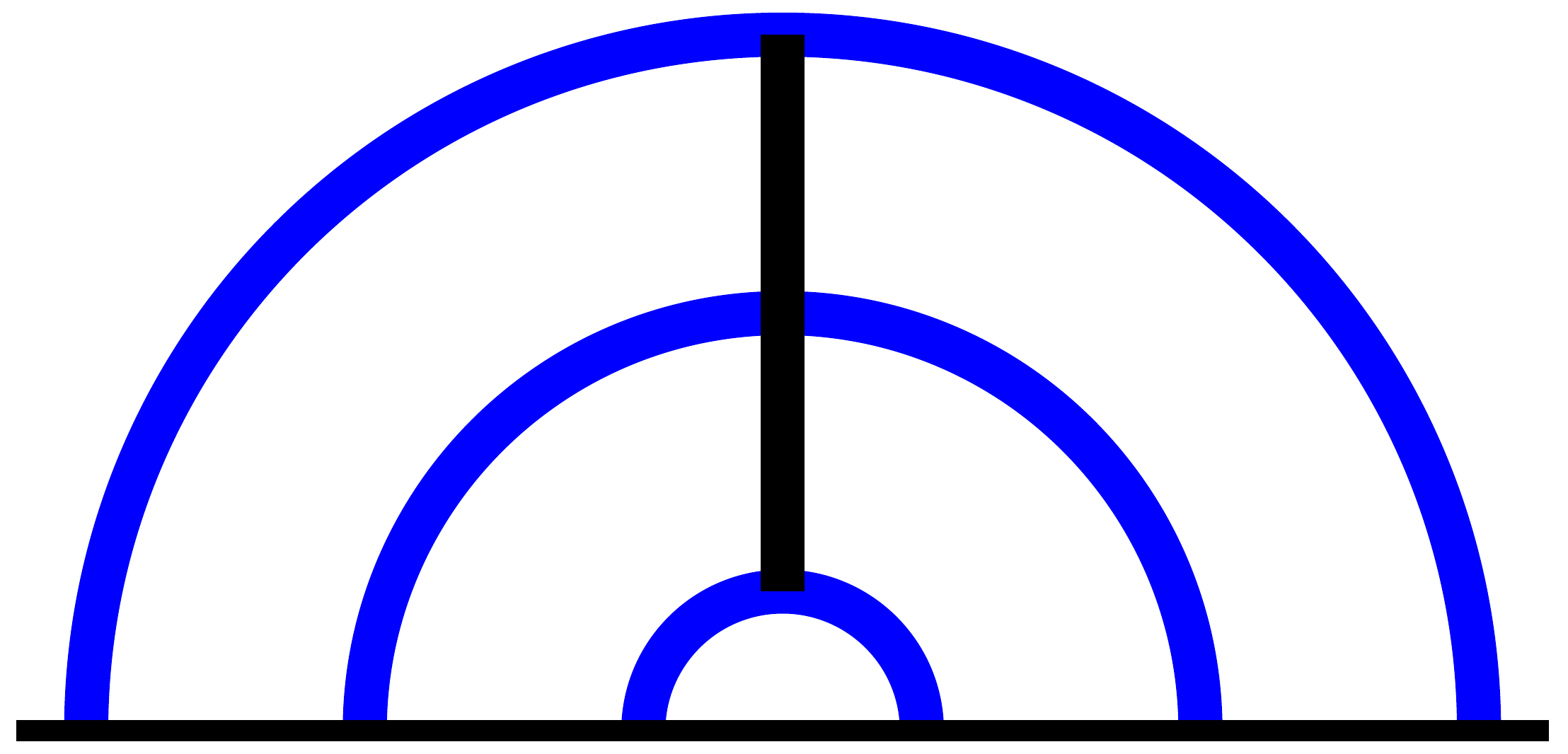}
\includegraphics[width=0.3\textwidth,height=0.15\textwidth]{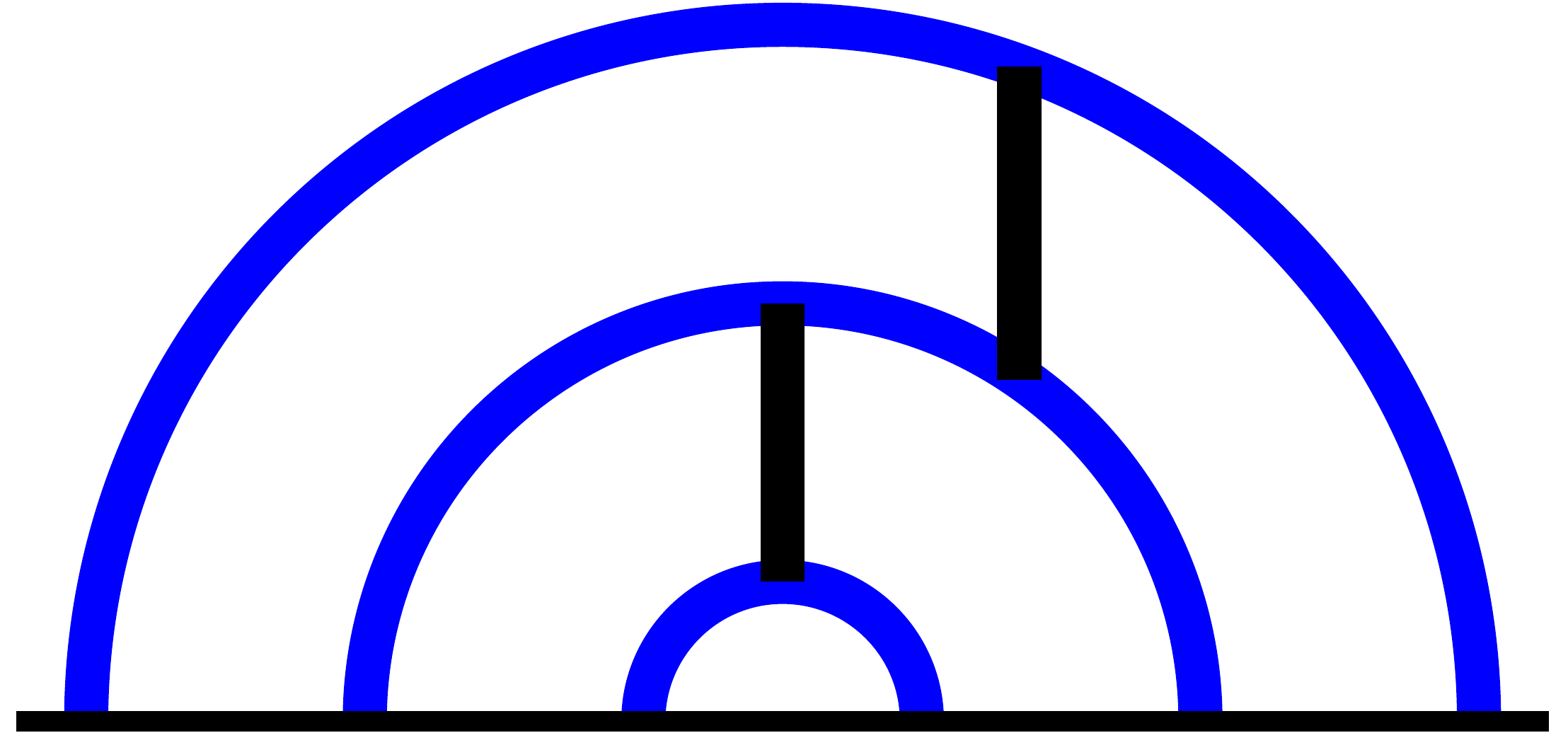}
\includegraphics[width=0.3\textwidth,height=0.15\textwidth]{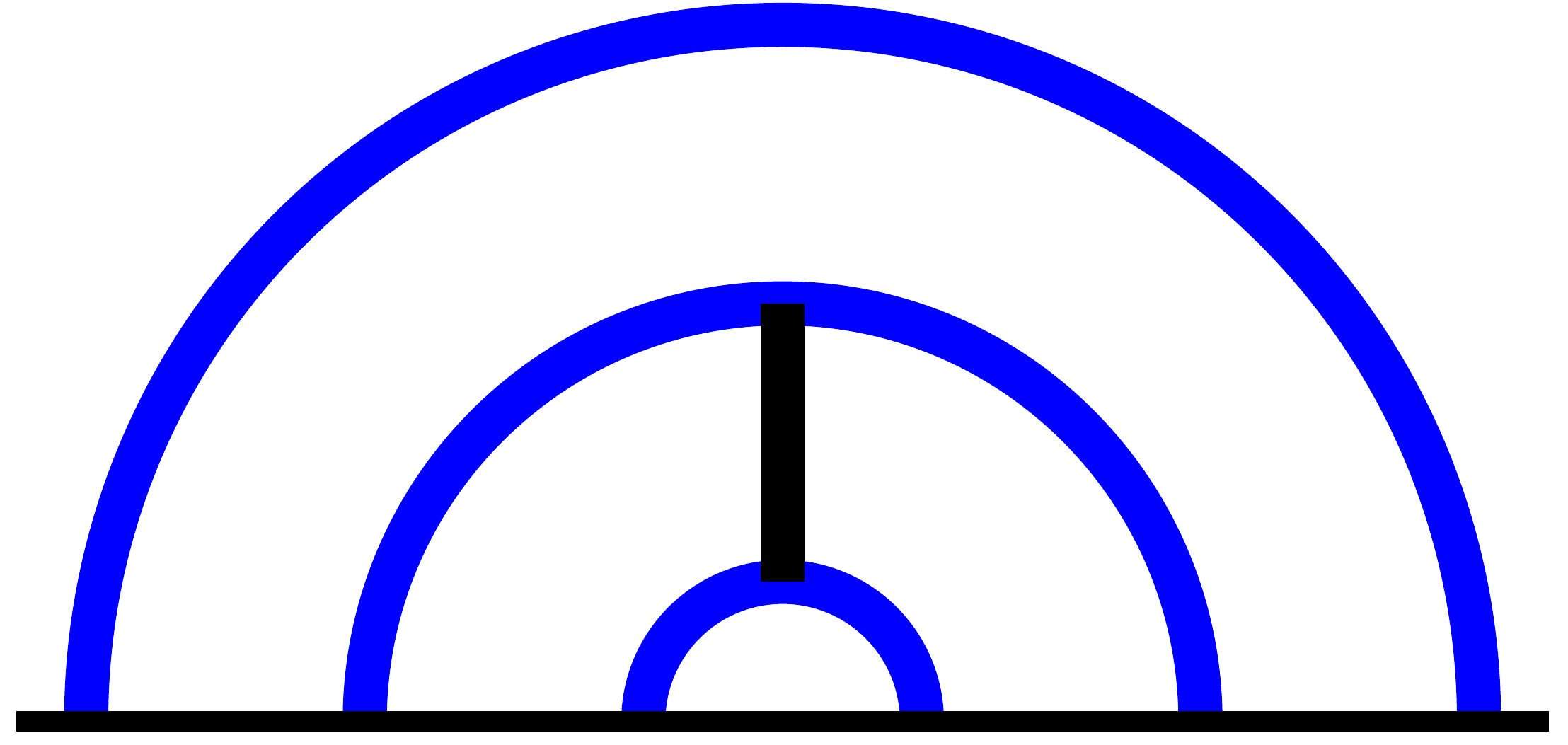}
\caption{Illustration for Example~\ref{Example:Connected}.}
\end{figure}

In view of \cite{Mlotkowski-Cumulants-Jacobi}, we have the following relation between
free cumulants and the Jacobi parameters of a given probability measure $\mu$:
\begin{equation}\label{freerm}
r_m=\sum_{(\sigma,\kappa)\in\mathrm{NCL}_{1,2}^{1}(m)}w(\sigma,\kappa),
\end{equation}
where $w(\sigma,\kappa):=\prod_{V\in\sigma}w(V,\kappa(V))$ and
\begin{equation}
w(V,k):=\left\{\begin{array}{ll}
\beta_k-\beta_{k-1}\phantom{ee}&\mbox{if $|V|=1$,}\\
\gamma_k-\gamma_{k-1}&\mbox{if $|V|=2$,}
\end{array}\right.
\end{equation}
under convention that $\beta_{-1}=\gamma_{-1}=0$.
In particular,
\begin{align}
r_1&=\beta_0,\label{freer1}\\
r_2&=\gamma_0,\label{freer2}\\
r_3&=\gamma_0(\beta_1-\beta_0),\label{freer3}\\
r_4&=\gamma_0\big[(\beta_1-\beta_0)^2+(\gamma_1-\gamma_0)\big],\label{freer4}\\
r_5&=\gamma_0\big[(\beta_1-\beta_0)^3
+3(\gamma_1-\gamma_0)(\beta_1-\beta_0)+\gamma_1(\beta_2-\beta_1)\big],\label{freer5}\\
r_6&=\gamma_0\big[(\beta_1-\beta_0)^4
+6(\gamma_1-\gamma_0)(\beta_1-\beta_0)^2\nonumber
+4\gamma_1(\beta_2-\beta_1)(\beta_1-\beta_0)\\
&\;\;\;\;\;+\gamma_1(\beta_2-\beta_1)^2+2(\gamma_1-\gamma_0)^2
+\gamma_1(\gamma_2-\gamma_1)\big].\label{freer6}
\end{align}

\begin{Thm}
\label{Thm:Linear-one-state}
Let $\mu$ be a probability measure with Jacobi parameters
$$
J({\mu})=\left(\begin{array}{ccccc}
\beta_0,&\beta_1,&\beta_2,&\beta_3,&\dots\\
\gamma_0,&\gamma_1,&\gamma_2,&\gamma_3,&\dots
\end{array}
\right).
$$

If for $t\in\mathbb{N}$, the Jacobi parameters for the free powers of $\mu$ are
$$
J({\mu}^{\boxplus t})=\left(\begin{array}{ccccc}
\beta_0(t),&\beta_1(t),&\beta_2(t),&\beta_3(t),&\dots\\
\gamma_0(t),&\gamma_1(t),&\gamma_2(t),&\gamma_3(t),&\dots
\end{array}
\right),
$$
and all the parameters $\beta_m(t)$, $\gamma_m(t)$ are polynomials on $t$,
then
\begin{equation}\label{freelinear}
\beta_1=\beta_2=\beta_3=\dots\quad\hbox{and}\quad
\gamma_1=\gamma_2=\gamma_3=\dots,
\end{equation}
so that $\mu$ is a free Meixner distribution.

On the other hand, if the Jacobi parameters of $\mu$
are such that (\ref{freelinear}) holds then,
putting
\begin{equation}
\label{Notation}
b:=\beta_1-\beta_0, \ c:=\gamma_1-\gamma_0,
\end{equation}
the measure $\mu_t = \mu^{\boxplus t}$ is well defined whenever $t\ge0$
and $c+t\gamma_0\ge0$ and then
\begin{equation}
\label{Free-powers}
J(\mu_t)=\left(\begin{array}{ccccc}
\beta_0 t,& b + \beta_0 t,& b + \beta_0 t,& b + \beta_0 t,& \dots\\
\gamma_0 t,& c + \gamma_0 t,& c + \gamma_0 t ,& c +\gamma_0 t,& \dots
\end{array}
\right).
\end{equation}
In particular, $\mu$ is $\boxplus$-infinitely divisible
if and only if $c \geq 0$.
In this case,
\begin{equation}\label{freelinearr}
R^{\mu}(z)= \beta_0 z +  \gamma_0 z^2\int_{\mathbb{R}}\frac{d\rho(x)}{1-zx},
\end{equation}
where
\begin{equation}
\label{jrho}
J({\rho})=\left(\begin{array}{ccccc}
b ,& b,& b,& b,&\dots\\
c ,& c,& c,& c,&\dots
\end{array}
\right),
\end{equation}
so that $\rho$ is a semicircular distribution.
\end{Thm}

\begin{proof}
Putting in formulas (\ref{freerm})--(\ref{freer6})
$r_m t$, $\beta_{m}(t)$, $\gamma_m(t)$, $w_t$ instead of
$r_m$, $\beta_m$, $\gamma_m$ and $w$ respectively we see
that
\[
\beta_0(t)=\beta_0 t, \quad\gamma_0(t)=\gamma_0 t.
\]
If $\gamma_0 = 0$, then $\mu$ is a point mass $\delta_{\beta_0}$. In this case we can take
\[
J({\mu}^{\boxplus t})=\left(\begin{array}{ccccc}
\beta_0 t,& 0,& 0,& \dots\\
0,& 0,& 0,& \dots
\end{array}
\right),
\]
which satisfies both the assumptions and the conclusions of the theorem.
From now on, we assume that $\gamma_0 > 0$.
From (\ref{freer3}) we get $r_3 t=\gamma_0 t(\beta_1(t)-\beta_0 t)$
which yields
\[
\beta_1(t)=\beta_1-\beta_0+\beta_0 t.
\]
Similarly, from (\ref{freer4}) we get
\[\gamma_1(t)=\gamma_1-\gamma_0+\gamma_0 t.\]

Now assume that all the Jacobi parameters $\beta_m(t)$, $\gamma_m(t)$ are
polynomials on $t$. We claim that $\beta_k(t)=\beta_1-\beta_0+\beta_0 t$
and $\gamma_k(t)=\gamma_1-\gamma_0+\gamma_0 t$ for all $k\ge1$.
Fix $d\ge2$ and assume that this holds for all $1\le k<d$.
Then for a block $V$, with $1\le|V|\le2$
and for $1<k<d$ we have $w_t(V,k)=0$. Now we consider
formula (\ref{freerm}) for $r_{2d+1}t$.

Put
\begin{align*}
\sigma_1&:=\big\{\{1,2d+1\},\{2,2d\},\{3,2d-1\},\dots,\{d,d+2\},\{d+1\}\big\},\\
\mathcal{K}_1&:=\{\kappa:(\sigma_1,\kappa)\in\mathrm{NCL}_{1,2}^{1}(2d+1),\,\,\kappa(\{d+1\})=d\}.
\end{align*}
Note that if $(\sigma,\kappa)\in\mathrm{NCL}_{1,2}^{1}(2d+1)\setminus(\{\sigma_1\}\times\mathcal{K}_1)$,
and $w_t(\sigma,\kappa)\ne0$
then $\kappa(V)=1$ for all $V\in\sigma\setminus\big\{\{1,2d+1\}\big\}$
and then
$$
w_t(\sigma,\kappa)=\gamma_0 t
\prod_{\substack{V\in\sigma\\|V|=1}}(\beta_1-\beta_0)\times
\prod_{\substack{V\in\sigma,\,\,|V|=2,\\V\ne\{1,2d+1\}}}(\gamma_1-\gamma_0).
$$
Now we observe that if $w_t(\sigma_1,\kappa)\ne0$, $\kappa\in\mathcal{K}_1$,
then $\kappa(V)\in\{0,1\}$ for all inner blocks $V\in\sigma_1$, $V\ne\{d+1\}$.
Therefore
\begin{align*}
\sum_{\kappa\in\mathcal{K}_1}w_t(\sigma_1,\kappa)
&=\gamma_0 t\big(\beta_{d}(t)-\beta_{d-1}(t)\big)
\prod_{k=1}^{d-1}\sum_{i=0}^{1}\left(\gamma_i(t)-\gamma_{i-1}(t)\right)\\
&=\gamma_0 t\big(\beta_d(t)-\beta_{d-1}(t)\big)\gamma_1(t)^{d-1},
\end{align*}
which is a polynomial of degree at least $d\ge2$, unless
$\beta_{d}(t)=\beta_{d-1}(t)=\beta_1-\beta_0+\beta_0 t$.

Now we consider the formula for $r_{2d+2}t$. Put
\begin{align*}
\sigma_2&:=\big\{\{1,2d+2\},\{2,2d+1\},\dots,\{d+1,d+2\}\big\},\\
\mathcal{K}_2&:=\{\kappa:(\sigma_2,\kappa)\in\mathrm{NCL}_{1,2}^{1}(2d+2),\,\, \kappa(\{d+1,d+2\})=d\}.
\end{align*}
Similarly as before, and using the previous step, we conclude that for
$(\sigma,\kappa)\in\mathrm{NCL}_{1,2}^{1}(2d+2)\setminus(\{\sigma_2\}\times\mathcal{K}_2)$
if we have $w_t(\sigma,\kappa)\ne0$ then $\kappa(V)=1$ for all
$V\in\sigma\setminus\big\{\{1,2d+2\}\big\}$ and then
$$
w_t(\sigma,\kappa)=\gamma_0 t
\prod_{\substack{V\in\sigma\\|V|=1}}(\beta_1-\beta_0)\times
\prod_{\substack{V\in\sigma,\,\,|V|=2,\\V\ne\{1,2d+2\}}}(\gamma_1-\gamma_0).
$$

Now we observe that if $w_t(\sigma_2,\kappa)\ne0$, $\kappa\in\mathcal{K}_2$,
then $\kappa(V)\in\{0,1\}$ for all inner blocks $V\in\sigma$, $V\ne\{d+1,d+2\}$.
Therefore
\begin{align*}
\sum_{\kappa\in\mathcal{K}_2}w_t(\sigma_2,\kappa)
&=\gamma_0 t\big(\gamma_{d}(t)-\gamma_{d-1}(t)\big)
\prod_{k=1}^{d-1}\sum_{i=0}^{1}\left(\gamma_i(t)-\gamma_{i-1}(t)\right)\\
&=\gamma_0 t\big(\gamma_d(t)-\gamma_{d-1}(t)\big)\gamma_1(t)^{d-1},
\end{align*}
which is again a polynomial of degree at least $d\ge2$, unless
$\gamma_{d}(t)=\gamma_{d-1}(t)=\gamma_1-\gamma_0+\gamma_0 t$.
This proves the first part.

Conversely, suppose that $\mu$ has Jacobi parameters given by equation~\eqref{freelinear}. As we have already noticed, if $(\sigma,\kappa)\in\mathrm{NCL}_{1,2}^{1}(m)$,
$V\in\sigma$ and $w(V,\kappa(V))\ne0$ then either $V$ is the outer block
of sigma (so $\kappa(V)=0$) or $\kappa(V)=1$. Then, using notation~\eqref{Notation}, for $\mu_t$ defined via equation \eqref{Free-powers}, we have $\mu_1 = \mu$, $r_1(\mu_t)=\beta_0 t$ and for $m\ge0$
\begin{equation}
\label{free-cumulant-product}
r_{m+2}(\mu_t)= t\cdot\gamma_0 \sum_{\sigma\in\mathrm{NC}_{1,2}(m)}\prod_{\substack{V\in\sigma\\|V|=1}} b
\cdot\prod_{\substack{V\in\sigma\\|V|=2}} c
\end{equation}
so that $r_m(\mu_t)=t\cdot r_m(\mu)$. Therefore $\mu_t = \mu^{\boxplus t}$.

Finally, it follows from equation \eqref{free-cumulant-product} that
 if $c \geq 0$ then $r_{m+2}=\gamma_0\cdot s_m(\rho)$
where the measure $\rho$ is defined by (\ref{jrho}), which proves (\ref{freelinearr})
(c.f. \cite{SaiConstant}).
\end{proof}

\section{The classical convolution.}

In \cite{Hasebe-Saigo-Monotone-cumulants}, the authors define cumulants $\set{r^\star_n}$ for an abstract convolution operation $\star$ via the following three properties. All convolutions and cumulants considered in this paper satisfy them (with the exception of the conditionally free cumulants; these latter satisfy properties which are similar to the ones below, but we choose instead to treat them separately in the next section).

\begin{enumerate}
\item
$r^\star_n(\mu^{\star t}) = t \cdot  r^\star_n(\mu)$, where $t \in \mf{N}$.
\item
$r^\star_n(D_\lambda \mu) = \lambda^n r^\star_n(\mu)$,
where $(D_\lambda \mu)(A) = \mu(\lambda^{-1} A)$.
\item
For any $n$, there is a polynomial $Q_n$ in $n-1$ variables such that
\[
s_n(\mu) = r^\star_n(\mu) + Q_n(r^\star_1(\mu), \ldots, r^\star_{n-1}(\mu)).
\]
\end{enumerate}

\begin{Lemma}
\label{Lemma:Initial-cumulants}
Suppose measures $\set{\mu_t}$ form a convolution semigroup with respect to an operation $\star$, with corresponding cumulants $\set{r^\star_k(t) = t \cdot r_k}$ satisfying axioms above. Then
\begin{enumerate}
\item
Denoting $s_i(t) = s_i(\mu_t)$ and $r_i = r_i^\star(\mu_1)$,
\begin{align*}
s_1(t) & = r_1 t, \\
s_2(t) & = r_2 t + a_{11} r_1^2 t^2, \\
s_3(t) & = r_3 t + a_{21} r_2 r_1 t^2 + a_{111} r_1^3 t^3, \\
s_4(t) & = r_4 t + a_{31} r_3 r_1 t^2 + a_{22} r_2^2 t^2 + a_{211} r_2 r_1^2 t^3 + a_{1111} r_1^4 t^4.
\end{align*}
\item
Suppose $a_{11} = 1$. Then
\begin{align*}
\beta_0(t) & = \beta_0 t, \\
\gamma_0(t) & = \gamma_0 t.
\end{align*}
\item
Suppose $\mu$ is not a point mass so that $\gamma_0 \neq 0$, and
\begin{equation}
\label{Coefficients}
a_{11} = a_{111} = a_{1111} = 1, \quad a_{211} = a_{21}^2 - 2 a_{21} + 3.
\end{equation}
Then
\begin{align*}
\beta_1(t) & = b + (a_{21} - 2) \beta_0 t, \\
\gamma_1(t) & = \left( b' - b^2 \right) + \left((a_{31} - 2 a_{21} + 2) b \beta_0 + (a_{22} - 1) \gamma_0 \right) t,
\end{align*}
an so all of these are linear in $t$.
\end{enumerate}
\end{Lemma}

\begin{proof}
Part (a) follows by combining properties (a-c) of the cumulants. The rest follow by combining part (a) with the formula~\eqref{accardiozejko}.
\end{proof}

\begin{Prop}
\label{Prop:General-Meixner}
In the setting of the preceding lemma, suppose $a_{11} = 1$. Let $\set{\mu_t}$, $\set{\mu_t'}$ be two convolution semigroups all of whose Jacobi parameters are polynomial in $t$. Suppose moreover that $\beta_0(\mu_t) = \beta_0(\mu_t')$, $\gamma_0(\mu_t) = \gamma_0(\mu_t')$, $\beta_1(\mu_t) = \beta_1(\mu_t')$, $\gamma_1(\mu_t) = \gamma_1(\mu_t')$, and that $\gamma_1(\mu_t)$ varies with $t$. Then $\mu_t$ and $\mu_t'$ have the same moments.
\end{Prop}

\begin{proof}
We will prove by induction on $n$ that the rest of the Jacobi parameters of the measure $\mu_t$ are the same as for $\mu_t'$. Suppose this is true up to $n-1$. Then using formula~\eqref{accardiozejko},
\[
s_{2n+1}(\mu_t) = s_{2n+1}(\mu_t') + (\beta_n(\mu_t) - \beta_n(\mu_t')) \gamma_{n-1}(\mu_t) \ldots \gamma_1(\mu_t) \gamma_0(\mu_t),
\]
so if
\[
s_k = r^\star_k + Q_k(r^\star_1, r^\star_2, \ldots, r^\star_{k-1}),
\]
then
\[
r^\star_{2n+1}(\mu_t) = r^\star_{2n+1}(\mu_t') + (\beta_{n}(\mu_t) - \beta_{n}(\mu_t')) \gamma_{n-1}(\mu_t) \ldots \gamma_1(\mu_t) \gamma_0(\mu_t).
\]
By Lemma~\ref{Lemma:Initial-cumulants}(b), $\gamma_0(\mu_t) = \gamma_0 t$, and $\gamma_1(\mu_t)$ has degree at least one in $t$, while the other $\gamma_i(\mu_t)$ are polynomial in $t$, from which it follows that $\beta_n(\mu_t) = \beta_n(\mu_t')$.

Similarly, from
\[
s_{2n+2}(\mu_t) = s_{2n+2}(\mu_t') + (\gamma_{n}(\mu_t) - \gamma_{n}(\mu_t')) \gamma_{n-1}(\mu_t) \ldots \gamma_1(\mu_t) \gamma_0(\mu_t),
\]
it follows that $\gamma_{n}(\mu_t) = \gamma_{n}(\mu_t)$.
\end{proof}

\begin{Cor}
\label{Cor:Meixner}
The Meixner distributions are the only convolution semigroups whose Jacobi parameters are polynomial in the convolution parameter.
\end{Cor}

\begin{proof}
If $\gamma_0 = 0$, the measure is a point mass, and so belongs to the Meixner class. So suppose $\gamma_0 \neq 0$. For the usual cumulants, $a_{21} = 3$, $a_{211} = 6$, $a_{22} = 3$, $a_{31} = 4$, so condition \eqref{Coefficients} is satisfied. Also, $\beta_1(t) = b + \beta_0 t$, and $\gamma_1(t) = 2 c + 2 \gamma_0 t$ is not constant. If $c \geq 0$, $t \geq 0$, or $c < 0$, $N = -\gamma_0/c \in \mf{N}$, $t \in \mf{N}$, the Meixner distribution with these initial values of the Jacobi parameters has
\[
J(\mu^{\ast t}) =
\begin{pmatrix}
\beta_0 t, & b + \beta_0 t, & 2 b + \beta_0 t, & 3 b + \beta_0 t, & \ldots \\
\gamma_0 t, & 2(c + \gamma_0 t), & 3 (2 c + \gamma_0 t), & 4 (3 c + \gamma_0 t), & \ldots
\end{pmatrix}
\]
Finally, suppose $c < 0$, $-\gamma_0/c \not \in \mf{N}$. Jacobi parameters in the preceding equation still define, via relation~\eqref{accardiozejko}, a linear functional on polynomials, which however is no longer positive. The uniqueness result in Proposition~\ref{Prop:General-Meixner} applies equally well to such functionals. Therefore, there is no positive linear functional (and so no measure) with these initial Jacobi parameters whose Jacobi parameters are polynomial in $t$.
\end{proof}

\begin{Example}
For the free cumulants, $a_{21} = 3$, $a_{211} = 6$, $a_{22} = 2$, $a_{31} = 4$, so condition \eqref{Coefficients} is satisfied. In this case, $\beta_1(t) = b + \beta_0 t$, and $\gamma_1(t) = c + \gamma_0 t$ is not constant. Therefore we re-prove our result that the free Meixner distributions are the only free convolution semigroups whose Jacobi parameters are polynomial in the convolution parameter.
\end{Example}

\begin{Example}
\label{Example:Boolean}
For the Boolean cumulants, $a_{21} = 2$, $a_{211} = 3$, $a_{22} = 1$, $a_{31} = 2$. So condition \eqref{Coefficients} is still satisfied, but $\gamma_1(t) = c$ is independent of $t$. Therefore Proposition~\ref{Prop:General-Meixner} does not apply. In fact, for \emph{any} Boolean convolution semigroup, the Jacobi parameters are
\[
J(\mu^{\uplus t}) =
\begin{pmatrix}
\beta_0 t, & \beta_1, & \beta_2, & \ldots \\
\gamma_0 t, & \gamma_1, & \gamma_2, & \ldots
\end{pmatrix}
\]
and so are polynomial in $t$, see \cite{Boz-Wys} and \cite{AnsBoolean}. See Proposition~\ref{Prop:Boolean} for a generalization of this result.
\end{Example}

\begin{Example}
For the monotone cumulants \cite{Muraki-Brownian-motion,Hasebe-Saigo-Monotone-cumulants}, $a_{11} = 1$, $a_{21} = \frac{5}{2}$, $a_{211} = \frac{13}{3}$. So Proposition~\ref{Prop:General-Meixner} applies, but condition \eqref{Coefficients} is not satisfied. In fact, in this case it is not clear if we have a $4$-parameter family of measures with linear Jacobi parameters; for example, this condition implies that the mean $\beta_0 = 0$. On the other hand, it may also be possible to have a monotone convolution semigroup with Jacobi parameters polynomial in $t$ of degree greater than $1$.
\end{Example}

\begin{Cor}
\label{Cor:Jacobi}
The only convolution semigroups $\set{\mu_t}$ which are spectral measures of tridiagonal matrices in equation~\eqref{Jacobi-matrix} are Meixner families.
\end{Cor}

\begin{proof}
$\mu_t$ is a spectral measure of a tridiagonal matrix in equation~\eqref{Jacobi-matrix} if and only if there are polynomials $\set{Q_n}$ orthogonal with respect to $\mu_t$ and satisfying the recursion relation
\[
x P_n(x) = (\alpha_n t + a_n) P_{n+1}(x) + (\beta_n t + b_n) P_n(x) + (\gamma_{n-1} t + c_{n-1}) P_{n-1}(x).
\]
It is well known that the corresponding monic orthogonal polynomials then satisfy the recursion
\[
x \hat{P}_n(x) = \hat{P}_{n+1}(x) + (\beta_n t + b_n) \hat{P}_n(x) + (\alpha_{n-1} t + a_{n-1}) (\gamma_{n-1} t + c_{n-1}) \hat{P}_{n-1}(x).
\]
By Corollary~\ref{Cor:Meixner}, $\mu_t$ is a Meixner distribution. A posteriori, for each $n$, $\alpha_{n-1} \gamma_{n-1} = 0$.
\end{proof}

\section{The two-state free convolution.}

Suppose that in addition to the measure $\mu$, we also have a measure $\widetilde{\mu}$, with moments $\widetilde{s}_m$
and Jacobi parameters $\widetilde{\gamma}_m,$ $\widetilde{\beta}_m$.
Recall from Section~\ref{Subsec:cumulants} that the \textit{conditionally free cumulants}
$R_m=R_{m}(\widetilde{\mu},\mu)$ of the pair $(\widetilde{\mu},\mu)$
are defined by
\begin{equation}
\widetilde{s}_m=
\sum_{\pi\in\mathrm{NC}(m)}\prod_{U\in\mathrm{Out}(\pi)}R_{|U|}(\widetilde{\mu},\mu)
\prod_{U\in\mathrm{Inn}(\pi)}r_{|U|}(\mu),
\end{equation}
where $r_m(\mu)$ are the free cumulants of $\mu$.

For $\sigma\in\mathrm{NC}_{1,2}(X)$ and $V\in\sigma$, with label $k$, we define
\begin{equation}
\label{w-conditionally}
\widetilde{w}(V,k,\sigma):=
\left\{\begin{array}{ll}
\phantom{}\widetilde{\beta}_k-\beta_{k-1}\phantom{ee}
&\mbox{if $|V|=1$ and $k=d(V,\sigma)$,}\\
\phantom{}\beta_k-\beta_{k-1}&\mbox{if $|V|=1$ and $k<d(V,\sigma)$,}\\
\phantom{}\widetilde{\gamma}_k-\gamma_{k-1}&\mbox{if $|V|=2$ and $k=d(V,\sigma)$,}\\
\phantom{}\gamma_k-\gamma_{k-1}&\mbox{if $|V|=2$ and $k<d(V,\sigma)$,}
\end{array}\right.
\end{equation}
keeping our convention that $\beta_{-1}=\gamma_{-1}=0$.
For $(\sigma,\kappa)\in\mathrm{NCL}_{1,2}(X)$ we put
\begin{equation}\label{cfreerm}
\widetilde{w}(\sigma,\kappa):=\prod_{V\in\sigma}\widetilde{w}(V,\kappa(V),\sigma).
\end{equation}
Then, in view of \cite{Mlotkowski-Cumulants-Jacobi}, for every $m\ge1$ we have
\begin{equation}\label{cfreerm2}
R_m(\widetilde{\mu},\mu)=\sum_{(\sigma,\kappa)\in\mathrm{NCL}_{1,2}^{1}(m)}
\widetilde{w}(\sigma,\kappa).
\end{equation}
For example:
\begin{align}
R_1&=\widetilde{\beta}_0,\label{cfreer1}\\
R_2&=\widetilde{\gamma}_0,\label{cfreer2}\\
R_3&=\widetilde{\gamma}_0\big(\widetilde{\beta}_1-\beta_0\big),\label{cfreer3}\\
R_4&=\widetilde{\gamma}_0\big[\big(\widetilde{\beta}_1-\beta_0\big)^2
+\big(\widetilde{\gamma}_1-\gamma_0\big)\big],\label{cfreer4}\\
R_5&=\widetilde{\gamma}_0\big[\big(\widetilde{\beta}_1-\beta_0\big)^3
+2\big(\widetilde{\gamma}_1-\gamma_0\big)\big(\widetilde{\beta}_1-\beta_0\big)
+\big(\widetilde{\gamma}_1-\gamma_0\big)\big(\beta_1-\beta_0\big)\nonumber\\
&\;\;\;\;+\widetilde{\gamma}_1\big(\widetilde{\beta}_2-\beta_1\big)
\big],\label{cfreer5}\\
R_6&=\widetilde{\gamma}_0\big[\big(\widetilde{\beta}_1-\beta_0\big)^4
+3\big(\widetilde{\gamma}_1-\gamma_0\big)
\big(\widetilde{\beta}_1-\beta_0\big)^2
+2\big(\widetilde{\gamma}_1-\gamma_0\big)\big(\widetilde{\beta}_1-\beta_0\big)
\big(\beta_1-\beta_0\big)\nonumber\\
&\;\;\;\;+\big(\widetilde{\gamma}_1-\gamma_0\big)\big(\beta_1-\beta_0\big)^2
+2\widetilde{\gamma}_1\big(\widetilde{\beta}_2-\beta_1\big)
\big(\widetilde{\beta}_1-\beta_0\big)
+2\widetilde{\gamma}_1\big(\widetilde{\beta}_2-\beta_1\big)
\big(\beta_1-\beta_0\big)\nonumber\\
&\;\;\;\;+\widetilde{\gamma}_1\big(\widetilde{\beta}_2-\beta_1\big)^2
+\big(\widetilde{\gamma}_1-\gamma_0\big)^2
+\big(\widetilde{\gamma}_1-\gamma_0\big)\big({\gamma}_1-\gamma_0\big)
+\widetilde{\gamma}_1\big(\widetilde{\gamma}_2-\gamma_1\big)
\big].\label{cfreer6}
\end{align}

The conditionally free power of a pair of measures:
$(\widetilde{\mu},\mu)^{\boxplus_c t}=(\widetilde{\mu}_t,\mu_t)$
is defined by: $\mu_t=\mu^{\boxplus t}$ and
$R_m(\widetilde{\mu}_t,\mu_t)=t\cdot R_m(\widetilde{\mu},\mu)$.

\begin{Thm}
\label{Thm:Linear-two-state}
Let $(\widetilde{\mu},\mu)$ be a pair of measures
with Jacobi parameters $\widetilde{\beta}_m,\widetilde{\gamma}_m$ and $\beta_m,\gamma_m$ respectively.

Assume that neither $\widetilde{\mu}$ nor $\mu$ is a point mass and that for the conditionally free powers
$\left(\widetilde{\mu}_t,\mu_t\right):=\left(\widetilde{\mu},\mu\right)^{\boxplus_c t}$, $t\in\mathbb{N}$,
all the Jacobi parameters of $\widetilde{\mu}_t$ are polynomials on $t$.
Then
\begin{align}
\beta_1&=\widetilde{\beta}_2=\beta_2=\widetilde{\beta}_3=\beta_3=\widetilde{\beta}_4=\ldots,
\label{cfreelinear1}\\
\gamma_1&=\widetilde{\gamma}_2=\gamma_2=\widetilde{\gamma}_3=\gamma_3=\widetilde{\gamma}_4=\ldots,
\label{cfreelinear2}
\end{align}
so that $\widetilde{\mu}$ is a general measure whose Jacobi parameters do not depend on $n$ for $n \geq 2$, and $\mu$ is the corresponding free Meixner distribution.

On the other hand, if (\ref{cfreelinear1}) and (\ref{cfreelinear2}) hold then,
putting
\begin{equation}
\label{Notation-tilde}
\widetilde{b}=\widetilde{\beta}_1-\beta_0, \ b=\beta_1-\beta_0, \ \widetilde{c}=\widetilde{\gamma}_1-\gamma_0, \ c=\gamma_1-\gamma_0,
\end{equation}
the conditionally free power
$\left(\widetilde{\mu}_t,\mu_t\right):=\left(\widetilde{\mu},\mu\right)^{\boxplus_c t}$
exists for $t\ge0$, ${c}+\gamma_0 t\ge0$, $\widetilde{c} + \gamma_0 t\ge0$
and we have
\begin{equation}
\label{General-mu}
J(\widetilde{\mu}_t)=\left(\begin{array}{ccccc}
\widetilde{\beta}_0 t,&\widetilde{b} + \beta_0 t,& b + \beta_0 t,& b + \beta_0 t,&\dots\\
\widetilde{\gamma}_0 t,&\widetilde{c} + \gamma_0 t,&c + \gamma_0 t,& c + \gamma_0 t,&\dots
\end{array}
\right)
\end{equation}
and
\begin{equation}
\label{General-nu}
J({\mu}_t)=\left(\begin{array}{ccccc}
\beta_0 t,& b + \beta_0 t,& b + \beta_0 t,& b + \beta_0 t,&\dots\\
\gamma_0 t,& c + \gamma_0 t,& c+ \gamma_0 t,& c + \gamma_0 t,&\dots
\end{array}
\right).
\end{equation}
In particular, the pair $(\widetilde{\mu},\mu)$ is $\boxplus_c$-infinitely divisible if and only if $c \geq 0$ and $\widetilde{c} \geq 0$. In this case,
\begin{equation}\label{cfreelinearr}
R^{\widetilde{\mu},\mu}(z)=\widetilde{\beta}_0 z+\widetilde{\gamma}_0 z^2\int_{\mathbb{R}} \frac{d\widetilde{\rho}(x)}{1-xz},
\end{equation}
where $\widetilde{\rho}$ is the free Meixner probability measure which satisfies
\begin{equation}\label{rhotilde}
J(\widetilde{\rho})=\left(\begin{array}{ccccc}
\widetilde{b},& b,& b,& b,&\dots\\
\widetilde{c},& c,& c,& c,&\dots
\end{array}
\right).
\end{equation}
\end{Thm}

\begin{Remark}
Note that we did not need to assume that the Jacobi parameters of ${\mu}_t$ are polynomials in $t$; rather, this fact is implied by the hypothesis of the theorem. If $\widetilde{\mu}$ is a point mass, the conclusion of the theorem holds if we also suppose that the Jacobi parameters of $\mu_t$ are polynomials in $t$. This follows from Theorem~\ref{Thm:Linear-one-state} and the fact that for any free convolution semigroup $\set{\mu_t}$, the family
\[
\set{(\widetilde{\mu}_t =\delta_{\widetilde{\beta}_0 t}, \mu_t)}
\]
form a two-state free convolution semigroup (with $R^{\widetilde{\mu},\mu}(z) = \widetilde{\beta}_0 z$).

If $\mu$ is a point mass, the conclusion of the theorem is false, see Example~\ref{Example:Boolean}. Proposition~\ref{Prop:Boolean} provides a complete description of this case.
\end{Remark}

\begin{Defn}
\emph{Two-state free Meixner distributions} are pairs of measures $(\widetilde{\mu}, \mu)$ with Jacobi parameters \eqref{General-mu} and \eqref{General-nu} for $t = 1$ and
\[
\gamma_0 > 0, \quad \widetilde{\gamma}_0 \geq 0, \quad c + \gamma_0 \ge 0, \quad \widetilde{c} + \gamma_0 \ge 0.
\]
\end{Defn}

\begin{Remark}
An explicit formula for $\widetilde{\mu}_t$ can be obtained from the continued fraction expansion of its Cauchy transform:
\[
G_{\widetilde{\mu}_t}(z) =
\cfrac{1}{z - \widetilde{\beta}_0 t -
\cfrac{\widetilde{\gamma}_0 t}{z - \beta_0 t - \widetilde{b} -
(\gamma_0 t + \widetilde{c})G_{\rho_t}(z)}},
\]
where $\rho_t$ is the semicircular distribution with mean $\beta_0 t + b$ and variance $\gamma_0 t + c$. The corresponding measure belongs to the Bernstein-Szeg{\H o} class, and has the form
\[
\widetilde{\mu}_t = \frac{\sqrt{4 (\gamma_0 t + c) - (x - \beta_0 t - b)^2}}{\text{cubic polynomial}} \,dx + \text{ at most $3$ atoms}.
\]
\end{Remark}

\begin{proof}[Proof of Theorem~\ref{Thm:Linear-two-state}]
Denote by $\widetilde{\beta}_m(t)$, $\widetilde{\gamma}_{m}(t)$
and $\beta_m(t)$, $\gamma_{m}(t)$ the Jacobi parameters of $\widetilde{\mu}_t$
and $\mu_t$ respectively.
Putting in formulas (\ref{freerm})--(\ref{freer6})
and (\ref{cfreerm})--(\ref{cfreer6})
\[
r_m t, \quad R_m t, \quad \beta_{m}(t), \quad \gamma_m(t), \quad \widetilde{\beta}_m(t), \quad\widetilde{\gamma}_m(t), \quad w_t, \quad \widetilde{w}_t
\]
instead of
\[
r_m, \quad R_m, \quad \beta_m, \quad \gamma_m, \quad \widetilde{\beta}_m, \quad \widetilde{\gamma}_m, \quad w, \quad \widetilde{w}
\]
respectively we see
that
$$
\beta_0(t)=\beta_0 t,\quad\gamma_0(t)=\gamma_0 t,\quad
\beta_1(t)=\beta_1-\beta_0+\beta_0 t,\quad
\gamma_1(t)=\gamma_1-\gamma_0+\gamma_0 t,
$$
by (\ref{freer1})--(\ref{freer4}), and
$$
\widetilde{\beta_0}(t)=\widetilde{\beta}_0 t,\quad
\widetilde{\gamma_0}(t)=\widetilde{\gamma}_0 t,\quad
\widetilde{\beta}_1(t)=\widetilde{\beta}_1-\beta_0+\beta_0 t,\quad
\widetilde{\gamma}_1(t)=\widetilde{\gamma}_1-\gamma_0+\gamma_0 t,
$$
from (\ref{cfreer1})--(\ref{cfreer4}).
Now assume that neither $\widetilde{\mu}$ nor $\mu$ is a point mass
(i.e. $\widetilde{\gamma}_0>0$, $\gamma_0>0$). It then follows from these formulas that $\gamma_0(t)$, $\gamma_1(t)$, $\widetilde{\gamma}_0(t)$, and $\widetilde{\gamma}_1(t)$ are all polynomials of degree one in $t$.

Assume that all $\widetilde{\beta}_m(t)$ and $\widetilde{\gamma}_m(t)$
are polynomials on $t$.
If we apply the last formulas to (\ref{cfreer5}) then we get
$$
R_5 t=\mathrm{constant}\cdot t+\widetilde{\gamma}_0(t)\widetilde{\gamma}_1(t)
\big(\widetilde{\beta}_2(t)-\beta_1(t)\big).
$$
Since, by assumption, $\widetilde{\beta}_2(t)$ is a polynomial, this implies that
\[
\widetilde{\beta}_2(t)=\beta_1(t)=\beta_1-\beta_0+\beta_0 t
\]
(for otherwise the right hand side would be a polynomial of degree at least 2).
Then in (\ref{cfreer6}) we obtain
$$
R_6 t=\mathrm{constant}\cdot t+\widetilde{\gamma}_0(t)\widetilde{\gamma}_1(t)
\big(\widetilde{\gamma}_2(t)-\gamma_1(t)\big),
$$
which, in turn, yields
\[
\widetilde{\gamma}_2(t)=\gamma_1(t)=\gamma_1-\gamma_0+\gamma_0 t.
\]

Now we are going to prove by induction that for every $n\ge1$:
\begin{align}
\widetilde{\beta}_{n+1}(t)&=\beta_n(t)=\beta_1-\beta_0+\beta_0 t,\label{indbeta}\\
\widetilde{\gamma}_{n+1}(t)&=\gamma_n(t)=\gamma_1-\gamma_0+\gamma_0 t.\label{indgamma}
\end{align}

Fix $d\ge2$ and suppose that (\ref{indbeta})--(\ref{indgamma})
hold for all $n$ such that $1\le n<d$.
Now we consider (\ref{freerm}) for $r_{2d+1}t$. Put
\begin{align*}
\sigma_1^{d}&:=\big\{\{1,2d+1\},\{2,2d\},\{3,2d-1\},\dots,\{d,d+2\},\{d+1\}\big\},\\
\mathcal{K}_1^{d}&:=
\{\kappa:(\sigma_1^{d},\kappa)\in\mathrm{NCL}_{1,2}^{1}(2d+1),\,\,\kappa(\{d+1\})=d\}.
\end{align*}
By our assumptions, if $1<k<d$ then $w_t(V,k)=0$ for any block $V$, with $1\le|V|\le2$.
Therefore the right hand side of (\ref{freerm}) for $r_{2d+1}t$ involves only such
$(\sigma,\kappa)\in\mathrm{NCL}_{1,2}^1(2d+1)$
that either
$\kappa(V)=1$ for every inner block $V\in\sigma$
(and then $w_t(V,\kappa(V))=w_t(V,1)=\beta_1-\beta_0$ if $|V|=1$
or $\gamma_1-\gamma_0$ if $|V|=2$) or
$\sigma=\sigma_{1}^{d}$, $\kappa\in\mathcal{K}_1^{d}$ and
$\kappa(V)\in\{0,1\}$ for all inner blocks $V\in\sigma$, $V\ne\{d+1\}$.
Accordingly we get
$$
r_{2d+1}t=
\gamma_0 t\!\!\!\!\sum_{\sigma\in\mathrm{NC}_{1,2}(2d-1)}
\prod_{\substack{V\in\sigma\\ |V|=1}}(\beta_1-\beta_0)\times\!\!\!\!\!
\prod_{\substack{V\in\sigma\\ |V|=2\\ V\ne\{1,2d+1\}}}(\gamma_1-\gamma_0)+
\sum_{\kappa\in\mathcal{K}_1^{d}}w_t(\sigma_1^{d},\kappa)
$$
and
\begin{align*}
\sum_{\kappa\in\mathcal{K}_1^{d}}w_t(\sigma_1^{d},\kappa)&=
\gamma_0 t \big(\beta_{d}(t)-\beta_{d-1}(t)\big)
\prod_{k=1}^{d-1}\sum_{i=0}^{1}\big(\gamma_i(t)-\gamma_{i-1}(t)\big)\\
&=\gamma_0 t\big(\beta_d(t)-\beta_{1}(t)\big)\gamma_1(t)^{d-1}.
\end{align*}
This implies that
\begin{equation}\label{c1}
\big(\beta_{d}(t)-\beta_{1}(t)\big)\gamma_1(t)^{d-1}
=c_1
\end{equation}
for some constant $c_1$.

Now we consider (\ref{cfreerm2}) for $R_{2d+3}t$.
Put
\begin{align*}
\sigma_1^{d+1}&:=\big\{\{1,2d+3\},\{2,2d+2\},\{3,2d+1\},\dots,\{d+1,d+3\},\{d+2\}\big\},\\
\mathcal{K}_1^{d+1}&:=
\left\{\kappa:(\sigma_1^{d+1},\kappa)\in\mathrm{NCL}_{1,2}^{1}(2d+3),
\,\,\kappa(\{d+2\})=d+1\right\},\\
\mathcal{L}_1^{d+1}&:=
\left\{\kappa:(\sigma_1^{d+1},\kappa)\in\mathrm{NCL}_{1,2}^{1}(2d+3),
\,\,\kappa(\{d+2\})=d\right\}.
\end{align*}
By our assumption, if $\widetilde{w}_t(\sigma,\kappa)\ne0$ then either
$\kappa(V)=1$ for every inner block $V\in\sigma$
or
$\sigma=\sigma_{1}^{d+1}$, $\kappa\in\mathcal{K}_{1}^{d+1}$
 and
$\kappa(V)\in\{0,1\}$ for all inner blocks $V\in\sigma$, $V\ne\{d+2\}$
or
$\sigma=\sigma_{1}^{d+1}$, $\kappa\in\mathcal{L}_{1}^{d+1}$, $\kappa(\{2,2d+2\})=1$
and $\kappa(V)\in\{0,1\}$ for all inner blocks $V\in\sigma$, $V\ne\{d+2\},\{2,2d+2\}$.
Therefore we have
\begin{align*}
R_{2d+3}t=&\widetilde{\gamma_0}t
\cdot\!\!\!\!\!\!\sum_{\sigma\in\mathrm{NC}_{1,2}(2d+1)}\!\!
\prod_{\substack{V\in\sigma\\|V|=1\\ V\in\mathrm{Out}(\sigma)}}
\!\!\!\!\!(\widetilde{\beta_1}-\beta_0)
\cdot\!\!\!\!\!\!\prod_{\substack{V\in\sigma\\|V|=1\\ V\in\mathrm{Inn}(\sigma)}}
\!\!\!\!(\beta_1-\beta_0)
\cdot\!\!\!\!\!\!\prod_{\substack{V\in\sigma\\|V|=2\\V\in\mathrm{Out}(\sigma)}}
\!\!\!\!(\widetilde{\gamma}_1-\gamma_0)
\cdot\!\!\!\!\!\!\prod_{\substack{V\in\sigma\\|V|=2\\V\in\mathrm{Inn}(\sigma)}}
\!\!\!\!(\gamma_1-\gamma_0)\\
&+\sum_{\kappa\in\mathcal{L}_1^{d+1}}\widetilde{w}_t(\sigma_1^{d+1},\kappa)+
\sum_{\kappa\in\mathcal{K}_1^{d+1}}\widetilde{w}_t(\sigma_1^{d+1},\kappa).
\end{align*}
In view of (\ref{c1}) we note that
\begin{align*}
\sum_{\kappa\in\mathcal{L}_1^{d+1}}\widetilde{w}_t(\sigma_1^{d+1},\kappa)&=
\widetilde{\gamma}_0 t\big(\widetilde{\gamma}_{1}(t)-\gamma_0(t)\big)
\big(\beta_d(t)-\beta_{d-1}(t)\big)
\widetilde{\gamma}_2(t)\ldots\widetilde{\gamma}_{d}(t)\\
&=\widetilde{\gamma}_0 t(\widetilde{\gamma}_{1}-\gamma_0)
\big(\beta_d(t)-\beta_{1}(t)\big)\gamma_1(t)^{d-1}\\
&=\widetilde{\gamma}_0 t(\widetilde{\gamma}_{1}-\gamma_0)c_1
\end{align*}
and that
\begin{align*}
\sum_{\kappa\in\mathcal{K}_1^{d+1}}&\widetilde{w}_t(\sigma_1^{d+1},\kappa)
=\widetilde{\gamma}_0 t\big(\widetilde{\beta}_{d+1}(t)-\beta_{d}(t)\big)
\widetilde{\gamma}_{1}(t)\widetilde{\gamma}_2(t)\ldots\widetilde{\gamma}_{d}(t)\\
&=\widetilde{\gamma}_0 t\big(\widetilde{\beta}_{d+1}(t)-\beta_{1}(t)\big)
\widetilde{\gamma}_{1}(t)\gamma_1(t)^{d-1}
+\widetilde{\gamma}_0 t\big(\beta_{1}(t)-\beta_{d}(t)\big)
\widetilde{\gamma}_{1}(t)\gamma_1(t)^{d-1}\\
&=\widetilde{\gamma}_0 t\big(\widetilde{\beta}_{d+1}(t)-\beta_{1}(t)\big)
\widetilde{\gamma}_{1}(t)\gamma_1(t)^{d-1}
-\widetilde{\gamma}_0 t\widetilde{\gamma}_{1}(t)c_1.
\end{align*}
If $\widetilde{\beta}_{d+1}(t)\ne\beta_{d-1}(t)$ then the first summand
is a polynomial of degree at least $d+1\ge3$ and if $c_1\ne0$
then the second one is a polynomial of degree $2$.
Therefore $c_1=0$ and
\begin{equation}\label{betad}
\widetilde{\beta}_{d+1}(t)=\beta_{d}(t)=\beta_1(t).
\end{equation}

Now we will study $\widetilde{\gamma}_{d+1}(t)$ and $\gamma_d(t)$
in a similar way.
Consider (\ref{freerm}) for $r_{2d+2}t$ and put
\begin{align*}
\sigma_2^{d}&:=\big\{\{1,2d+2\},\{2,2d+1\},\{3,2d\},\dots,\{d+1,d+2\}\big\},\\
\mathcal{K}_2^{d}&:=\{\kappa:(\sigma_2^{d},\kappa)\in\mathrm{NCL}_{1,2}^{1}(2d+2),
\,\,\kappa(\{d+1,d+2\})=d\}.
\end{align*}
Then by inductive assumption and by (\ref{betad}) we have
$$
r_{2d+2}t=
\gamma_0 t\!\!\!\!\sum_{\sigma\in\mathrm{NC}_{1,2}(2d)}
\prod_{\substack{V\in\sigma\\ |V|=1}}(\beta_1-\beta_0)\cdot\!\!
\prod_{\substack{V\in\sigma\\ |V|=2}}(\gamma_1-\gamma_0)+
\sum_{\kappa\in\mathcal{K}_2^{d}}w(\sigma_2^{d},\kappa)
$$
and similarly as before we see that
$$
\sum_{\kappa\in\mathcal{K}_2^{d}}w_t(\sigma_2^{d},\kappa)
=\gamma_0 t\big(\gamma_d(t)-\gamma_{1}(t)\big)\gamma_1(t)^{d-1},
$$
which implies that
\begin{equation}\label{c2}
\big(\gamma_d(t)-\gamma_{1}(t)\big)\gamma_1(t)^{d-1}=c_2
\end{equation}
for some constant $c_2$.

Now we consider (\ref{cfreerm2}) for $R_{2d+4}t$.
Put
\begin{align*}
\sigma_2^{d+1}&:=\big\{\{1,2d+4\},\{2,2d+3\},\{3,2d+2\},\dots,\{d+2,d+3\}\big\},\\
\mathcal{K}_2^{d+1}&:=
\left\{\kappa:(\sigma_2^{d+1},\kappa)\in\mathrm{NCL}_{1,2}^{1}(2d+4),
\,\,\kappa(\{d+2,d+3\})=d+1\right\},\\
\mathcal{L}_2^{d+1}&:=
\left\{\kappa:(\sigma_2^{d+1},\kappa)\in\mathrm{NCL}_{1,2}^{1}(2d+4),
\,\,\kappa(\{d+2,d+3\})=d\right\}.
\end{align*}
By our assumption and by (\ref{betad})
we have
\begin{align*}
R_{2d+4}t=&\widetilde{\gamma_0}t\cdot\!\!\!\!\!\!
\sum_{\sigma\in\mathrm{NC}_{1,2}(2d+2)}\!\!
\prod_{\substack{V\in\sigma\\|V|=1\\ V\in\mathrm{Out}(\sigma)}}
\!\!\!\!\!(\widetilde{\beta_1}-\beta_0)
\cdot\!\!\!\!\!\!\prod_{\substack{V\in\sigma\\|V|=1\\ V\in\mathrm{Inn}(\sigma)}}
\!\!\!\!(\beta_1-\beta_0)
\cdot\!\!\!\!\!\!\prod_{\substack{V\in\sigma\\|V|=2\\V\in\mathrm{Out}(\sigma)}}
\!\!\!\!(\widetilde{\gamma}_1-\gamma_0)
\cdot\!\!\!\!\!\!\prod_{\substack{V\in\sigma\\|V|=2\\V\in\mathrm{Inn}(\sigma)}}
\!\!\!\!(\gamma_1-\gamma_0)\\
&+\sum_{\kappa\in\mathcal{L}_2^{d+1}}\widetilde{w}_t(\sigma_2^{d+1},\kappa)+
\sum_{\kappa\in\mathcal{K}_2^{d+1}}\widetilde{w}_t(\sigma_2^{d+1},\kappa).
\end{align*}

Now similarly as before we note from (\ref{c2}) that
$$
\sum_{\kappa\in\mathcal{L}_1^{d+1}}\widetilde{w}_t(\sigma_1^{d+1},\kappa)=
t\widetilde{\gamma}_0 (\widetilde{\gamma}_{1}-\gamma_0)c_2,
$$
and that
$$
\sum_{\kappa\in\mathcal{K}_2^{d+1}}\widetilde{w}_t(\sigma_2^{d+1},\kappa)
=\widetilde{\gamma}_0 t \big(\widetilde{\gamma}_{d+1}(t)-\gamma_{1}(t)\big)
\widetilde{\gamma}_{1}(t)\gamma_1(t)^{d-1}
-\widetilde{\gamma}_0 t
\widetilde{\gamma}_{1}(t)c_2.
$$

If $\widetilde{\gamma}_{d+1}(t)\ne\gamma_{d-1}(t)$ then the first summand
is a polynomial of degree at least $d+1\ge3$ and if $c_2\ne0$
then the second is a polynomial of degree $2$.
Therefore $c_2=0$ and
\begin{equation}\label{gammad}
\widetilde{\gamma}_{d+1}(t)=\gamma_{d}(t)=\gamma_1(t),
\end{equation}
which completes the proof of the first part.

Conversely, suppose that $\widetilde{\mu}$ and $\mu$ (which in this case can be point masses) have Jacobi
parameters given by equations (\ref{cfreelinear1}) and (\ref{cfreelinear2}),
respectively. Then by Theorem~\ref{Thm:Linear-one-state}, \eqref{General-nu} holds for $\mu_t = \mu^{\boxplus t}$. Also, using notation~\eqref{Notation-tilde}, for $\widetilde{\mu}_t$ defined via equation \eqref{General-mu}, we have $\widetilde{\mu}_1 = \widetilde{\mu}$,
$R_1(\widetilde{\mu}_t,\mu_t)=\widetilde{\beta}_0 t$ and for $m\ge0$
\begin{equation}
\label{C-free-cumulant-expansion}
R_{m+2}(\widetilde{\mu}_t,\mu_t)=t\cdot\widetilde{\gamma_0}\!\!\!\!\!\!
\sum_{\sigma\in\mathrm{NC}_{1,2}(m)}\!\!
\prod_{\substack{V\in\sigma\\|V|=1\\ V\in\mathrm{Out}(\sigma)}}
\widetilde{b}
\cdot \prod_{\substack{V\in\sigma\\|V|=1\\ V\in\mathrm{Inn}(\sigma)}}
b
\cdot \prod_{\substack{V\in\sigma\\|V|=2\\V\in\mathrm{Out}(\sigma)}}
\widetilde{c}
\cdot \prod_{\substack{V\in\sigma\\|V|=2\\V\in\mathrm{Inn}(\sigma)}}
c,
\end{equation}
so that $R_m(\widetilde{\mu}_t,\mu_t)=t\cdot R_m(\widetilde{\mu},\mu)$. Therefore $\left(\widetilde{\mu}_t,\mu_t\right)=\left(\widetilde{\mu},\mu\right)^{\boxplus_c t}$.

Finally, it follows from equation \eqref{C-free-cumulant-expansion} that
 if $c \geq 0$ and $\widetilde{c} \geq 0$,
then $R_{m+2}(\widetilde{\mu},\mu)=\widetilde{\gamma}_0\cdot s_m(\widetilde{\rho})$,
which, in turn, yields (\ref{cfreelinearr}).
\end{proof}

\begin{Remark}
In addition to the convolution property with respect to the parameter $t$, note that for fixed $\widetilde{b},b,\widetilde{c},c$, the family of all two-state free Meixner distributions $(\widetilde{\mu},\mu)$ constitute a four-parameter $\boxplus_c$-semigroup with respect to the parameters
$\widetilde{\beta}_0, \beta_0$, $\widetilde{\gamma}_0,\gamma_0$ for $\gamma_0>0$, $\gamma_0\ge -c$, $\widetilde{\gamma}_0\ge \max(0, -\widetilde{c})$. Indeed, by formulas \eqref{freerm} and \eqref{cfreerm2} (see also \cite{Mlotkowski-Examples-conditional}), if $(\widetilde{\mu}',\mu')$ and $(\widetilde{\mu}'',\mu'')$ are two-state free Meixner distributions, with parameters $\widetilde{\beta}_0',\widetilde{b},\beta_0',b$, $\widetilde{\gamma}_0',\widetilde{c},\gamma_0',c$ and $\widetilde{\beta}_0'',\widetilde{b},\beta_0''$, $b,\widetilde{\gamma}_0'',\widetilde{c},\gamma_0'',c$ respectively, then $(\widetilde{\mu}',\mu')\boxplus_c(\widetilde{\mu}'',\mu'')$ is again a two-state free Meixner distribution with parameters
\[
\widetilde{\beta}_0'+\widetilde{\beta}_0'', \widetilde{b}, \beta_0'+\beta_0'', b, \widetilde{\gamma}_0'+\widetilde{\gamma}_0'', \widetilde{c}, \gamma_0'+\gamma_0'', c.
\]
\end{Remark}

\begin{Prop}
\label{Prop:Boolean}
Let $u\in\mathbb{R}$ and $\widetilde{\mu}$ be an arbitrary probability measure
on $\mathbb{R}$ with
\[
J(\widetilde{\mu})
=\begin{pmatrix}
\widetilde{\beta}_0, &\widetilde{\beta}_1,&\widetilde{\beta}_2,&\widetilde{\beta}_3,&\ldots \\
\widetilde{\gamma}_0,&\widetilde{\gamma}_1,&\widetilde{\gamma}_2,&\widetilde{\gamma}_3,&\ldots
\end{pmatrix}.
\]
For $t>0$ define $\widetilde{\mu}_t$ by
\[
J(\widetilde{\mu}_t)
=\begin{pmatrix}
\widetilde{\beta}_0 t, &\widetilde{\beta}_1+(t-1)u,&\widetilde{\beta}_2+(t-1)u,&
\widetilde{\beta}_3+(t-1)u,&\ldots \\
\widetilde{\gamma}_0 t,&\widetilde{\gamma}_1,&\widetilde{\gamma}_2,&\widetilde{\gamma}_3,&\ldots
\end{pmatrix}.
\]
Then the pair $(\widetilde{\mu},\delta_u)$ is $\boxplus_c$-infinitely divisible
and $\left\{\left(\widetilde{\mu}_t,\delta_{tu}\right)\right\}_{t>0}$ is the corresponding $\boxplus_c$-semigroup.

Moreover, we have
\[
R^{\widetilde{\mu},\delta_u}(z)=\widetilde{\beta}_0 z
+\widetilde{\gamma}_0 z^2\int_{\mathbb{R}}\frac{d\widetilde{\rho}(x)}{1-xz},
\]
where $\widetilde{\rho}$ is given by
\[
J(\widetilde{\rho})
=\begin{pmatrix}
\widetilde{\beta}_1-u, &\widetilde{\beta}_2-u,&\widetilde{\beta}_3-u,&\ldots \\
\widetilde{\gamma}_1,&\widetilde{\gamma}_2,&\widetilde{\gamma}_3,&\ldots
\end{pmatrix}.
\]
\end{Prop}

\begin{proof}
Since $J(\delta_{tu})=\begin{pmatrix}
tu, &0,&0,&0,&\ldots \\
0,&0,&0,&0,&\ldots\end{pmatrix}$,
in calculating $R_m(\widetilde{\mu}_t,\delta_{tu})$
we can restrict ourselves to those pairs
$(\sigma,\kappa)\in\mathrm{NCL}_{1,2}^{1}(m)$
that for $V\in\sigma$ we have
either $|V|=2$, $\kappa(V)=d(V,\sigma)$
or $|V|=1$, $k\in\{1,d(V,\sigma)\}$. More precisely,
for the pair $(\widetilde{\mu}_t,\delta_{tu})$ we have in (\ref{w-conditionally})
\begin{equation*}
\widetilde{w}(V,k,\sigma):=
\left\{\begin{array}{ll}
\phantom{}t\widetilde{\beta}_0&\mbox{if $|V|=1$, $k=d(V,\sigma)=0$,}\\
\phantom{}\widetilde{t\gamma}_0&\mbox{if $|V|=2$ and $k=d(V,\sigma)=0$,}\\
\phantom{}\widetilde{\beta}_1-u\phantom{ee}
&\mbox{if $|V|=1$ and $k=d(V,\sigma)=1$,}\\
\phantom{}\widetilde{\beta}_k+(t-1)u&\mbox{if $|V|=1$ and $k=d(V,\sigma)>1$,}\\
\phantom{}-tu&\mbox{if $|V|=1$, $k=1<d(V,\sigma)$,}\\
\phantom{}\widetilde{\gamma}_k&\mbox{if $|V|=2$ and $k=d(V,\sigma)>0$,}\\
\phantom{}0&\mbox{otherwise.}
\end{array}\right.
\end{equation*}
This implies that for fixed $\sigma\in\mathrm{NC}_{1,2}^{1}(m)$, $m\ge2$,
with only one outer block, we have
\[
\sum_{\substack{\kappa\\(\sigma,\kappa)\in\mathrm{NCL}_{1,2}^{1}(m)}}
\prod_{V\in\sigma}\widetilde{w}(V,\kappa(V),\sigma)
=t\widetilde{\gamma}_0
\prod_{\substack{V\in\mathrm{Inn}(\sigma)\\|V|=2}}\widetilde{\gamma}_{d(V,\sigma)}
\prod_{\substack{V\in\sigma\\|V|=1}}\left(\widetilde{\beta}_{d(V,\sigma)}-u\right).
\]
Hence $R_{1}(\widetilde{\mu}_t,\delta_{tu})=\widetilde{\beta}_0 t$ and for $m\ge0$
we can write
\[
R_{m+2}(\widetilde{\mu}_t,\delta_{tu})
=t\widetilde{\gamma}_0\sum_{\sigma\in\mathrm{NC}_{1,2}(m)}
\prod_{\substack{V\in\sigma\\|V|=2}}\widetilde{\gamma}_{d(V,\sigma)+1}
\prod_{\substack{V\in\sigma\\|V|=1}}\left(\widetilde{\beta}_{d(V,\sigma)+1}-u\right).
\]
In particular, $R_{m}(\widetilde{\mu}_t,\delta_{tu})=tR_{m}(\widetilde{\mu},\delta_{u})$
for all $m\ge1$.
\end{proof}

\section{The two-state free Meixner class}

Free Meixner distributions arise in many results in free and Boolean probability theories. In this section we describe a number of appearances of the family from Theorem~\ref{Thm:Linear-two-state} in the two-state-free probability theory, which justify the name ``two-state free Meixner class''. Other places where measures with Jacobi parameters independent of $n$ for $n \geq 2$ were encountered include Theorems 11 and 12 of \cite{Kry-Woj-Associative}, examples in \cite{Lenczewski-Decompositions-convolution}, as well as \cite{Hinz-Mlotkowski-Free-Cumulants,Hinz-Mlotkowski-Limit-conditinally} and \cite{Hamada-Konno-Mlotkowksi}.

\begin{Defn}
A triple $(\mc{A}, \phi, \psi)$ is an (algebraic) \emph{two-state non-commutative probability space} if $\mc{A}$ is a $\ast$-algebra, and $\phi, \psi$ are states (positive, unital linear functionals) on it. A self-adjoint ($X = X^\ast$) element $X \in \mc{A}$ has the \emph{distribution} $(\widetilde{\mu}, \mu)$ in $(\mc{A}, \phi, \psi)$ if $\widetilde{\mu}, \mu$ are probability measures such that
\[
\phi[X^n] = s_n(\widetilde{\mu}), \qquad \psi[X^n] = s_n(\mu)
\]
for all $n \geq 0$.
\end{Defn}

The next remark is the analog of the three results in Remark~\ref{Remark:Free-Meixner}(a).

\begin{Remark}(Limit theorems and specific distributions)
\begin{enumerate}
\item
Theorem~4.3 of \cite{BLS96} is the two-state free central limit theorem. Let $(\widetilde{\nu}, \nu)$ be a pair of measures such that
\[
s_1(\widetilde{\nu}) = 0 = s_1(\nu), \qquad s_2(\widetilde{\nu}) = v \geq 0, \qquad s_2(\nu) = u \geq 0.
\]
Then, denoting by $D$ the dilation operator, we have the weak limit.
\[
D_{1/\sqrt{N}} (\widetilde{\nu}, \nu)^{\boxplus_c N} \rightarrow (\widetilde{\mu}, \mu).
\]
The authors give explicit formulas for the limit distributions, based on the observation that
\[
R^\mu(z) = u z^2, \qquad R^{\widetilde{\mu}, \mu}(z) = v z^2.
\]
Thus formulas \eqref{freelinearr} and \eqref{cfreelinearr} give
\[
\widetilde{\beta}_0 = \beta_0 = 0, \quad \widetilde{\gamma}_0 = v, \quad \gamma_0 = u, \quad \widetilde{\rho} = \rho = \delta_0, \quad \widetilde{b} = b = 0, \quad \widetilde{c} = c = 0.
\]
In other words
\[
J(\mu) =
\begin{pmatrix}
0, & 0, & 0 & \ldots \\
u, & u, & u, & \ldots
\end{pmatrix}
\]
and $\mu$ is a semicircular distributions, while
\[
J(\widetilde{\mu}) =
\begin{pmatrix}
0, & 0, & 0 & \ldots \\
v, & u, & u, & \ldots
\end{pmatrix}
\]
and $\widetilde{\mu}$ is a symmetric free Meixner distribution.
\item
Theorem~4.4 of \cite{BLS96} is the two-state free Poisson limit theorem. Let $(\widetilde{\nu}_N, \nu_N)$ be
\[
\widetilde{\nu}_N = \left(1 - \frac{q}{N} \right) \delta_0 + \frac{q}{N} \delta_1, \qquad \nu_N = \left(1 - \frac{p}{N} \right) \delta_0 + \frac{p}{N} \delta_1.
\]
Then we have the weak limit
\begin{equation}
\label{Poisson-limit}
(\widetilde{\nu}_N, \nu_N)^{\boxplus N} \rightarrow (\widetilde{\mu}, \mu).
\end{equation}
Again, the starting point for computing the explicit limit densities is the observation that
\[
R^\mu(z) = \frac{p z}{1 - z}, \qquad R^{\widetilde{\mu}, \mu}(z) = \frac{q z}{1 - z},
\]
In this case formulas \eqref{freelinearr} and \eqref{cfreelinearr} give
\[
\widetilde{\beta}_0 = \widetilde{\gamma}_0 = q, \quad \beta_0 = \gamma_0 = p, \quad \widetilde{\rho} = \rho = \delta_1, \quad \widetilde{b} = b = 1, \quad \widetilde{c} = c = 0.
\]
In other words
\[
J(\mu) =
\begin{pmatrix}
p, & 1 + p, & 1 + p & \ldots \\
p, & p, & p, & \ldots
\end{pmatrix}
\]
and $\mu$ is a free Poisson distribution, while
\[
J(\widetilde{\mu}) =
\begin{pmatrix}
q, & 1 + p, & 1 + p, & \ldots \\
q, & p, & p, & \ldots
\end{pmatrix}
\]
is a free Meixner distribution.
\item
A more general version of the Poisson limit theorem (which appears to be new) is to take
\[
\widetilde{\nu}_N = \left(1 - \frac{q}{N} \right) \delta_0 + \frac{q}{N} \delta_v, \qquad \nu_N = \left(1 - \frac{p}{N} \right) \delta_0 + \frac{p}{N} \delta_u
\]
for $u, v \neq 0$. In this case the limit distributions in \eqref{Poisson-limit} satisfy
\[
R^\mu(z) = \frac{p u z}{1 - u z}, \qquad R^{\widetilde{\mu}, \mu}(z) = \frac{q v z}{1 - v z}.
\]
Thus
\[
\widetilde{\beta}_0 = q v, \quad \beta_0 = p u, \quad \widetilde{\gamma}_0 = q v^2, \quad \gamma_0 = p u^2, \quad \widetilde{\rho} = \delta_v, \quad \rho = \delta_u, \quad \widetilde{b} = v, \quad b = u, \quad \widetilde{c} = c = 0.
\]
In other words
\[
J(\mu) =
\begin{pmatrix}
p u, & u + p u, & u + p u & \ldots \\
p u^2, & p u^2, & p u^2, & \ldots
\end{pmatrix}
\]
and $\mu$ is still a free Poisson distribution, but
\[
J(\widetilde{\mu}) =
\begin{pmatrix}
q v, & v + p u, & u + p u, & u + p u, & \ldots \\
q v^2, & p u^2, & p u^2, & p u^2, & \ldots
\end{pmatrix}
\]
only has Jacobi parameters independent of $n$ for $n \geq 2$.
\item
Let $(\mc{A}, \phi, \psi)$ be a two-state non-commutative probability space, and $X$ a self-adjoint idempotent such that
\[
\phi[X] = q, \quad \psi[X] = p
\]
with $0 < p, q < 1$. Thus the distributions $\widetilde{\mu}$ and $\mu$ of $X$ with respect to $\phi$ and $\psi$ are both Bernoulli distributions, so we can refer to the pair $(\widetilde{\mu}, \mu)$ as a two-state Bernoulli distribution. We now note that these distributions (and so their convolution powers, the two-state free binomial distributions) are two-state free Meixner distributions. Indeed,
\[
\widetilde{\mu} = (1 - q) \delta_0 + q \delta_1, \qquad \mu = (1 - p) \delta_0 + p \delta_1.
\]
It is easy to see that we can write
\[
J(\mu) =
\begin{pmatrix}
p, & 1 - p \\
p (1-p), & 0
\end{pmatrix}
\]
and
\[
J(\widetilde{\mu}) =
\begin{pmatrix}
q , & 1 - q \\
q (1-q), & 0
\end{pmatrix}
\]
(see the comment about finitely supported measures in Section~\ref{Subsec:Jacobi}). Thus $(\widetilde{\mu}, \mu)$ is a two-state free Meixner distribution with
\[
\widetilde{\beta}_0 = q, \; \beta_0 = p, \quad \widetilde{\gamma}_0 = q (1-q), \; \gamma_0 = p (1-p), \quad \widetilde{b} = 1 - p - q, \; b = 1 - 2p, \quad \widetilde{c} = c = -p (1-p).
\]
Consequently, the two-state free binomial distributions $(\widetilde{\mu}_t, \mu_t) = (\widetilde{\mu}, \mu)^{\boxplus_c t}$ have
\[
J(\widetilde{\mu}) =
\begin{pmatrix}
q t, & 1 - p - q + p t, & 1 - 2 p + p t, & 1 - 2 p + p t, & \ldots \\
q (1-q) t, & p (1-p) (t-1), & p (1-p) (t-1), & p (1-p) (t-1), & \ldots
\end{pmatrix}
\]
and
\[
J(\mu) =
\begin{pmatrix}
p t, & 1 - 2 p + p t, & 1 - 2 p + p t, & 1 - 2 p + p t, & \ldots \\
p (1-p) t, & p (1-p) (t-1), & p (1-p) (t-1), & p (1-p) (t-1), & \ldots
\end{pmatrix}
\]
for $t \geq 1$.
\end{enumerate}
\end{Remark}

The next proposition is the analog of Remark~\ref{Remark:Free-Meixner}(d).

\begin{Prop}
\label{Prop:Quadratic-R-transform}
$(\widetilde{\mu}, \mu)$ is a two-state free Meixner distribution if and only if its two-state free cumulants satisfy the recursion
\begin{equation}
\label{Two-state-recursion}
R_{m+2}(\widetilde{\mu}, \mu) = \widetilde{b} \ R_{m+1}(\widetilde{\mu}, \mu) + \frac{\widetilde{c}}{\gamma_0} \ \sum_{k=2}^{m} r_k(\mu) \ R_{m+2-k}(\widetilde{\mu}, \mu),
\end{equation}
(with $\gamma_0 \neq 0$) with the initial conditions
\[
R_1(\widetilde{\mu}, \mu) = \widetilde{\beta}_0, \qquad R_2(\widetilde{\mu}, \mu) = \widetilde{\gamma}_0,
\]
and the free cumulants of $\mu$ satisfy the recursion
\begin{equation}
\label{One-state-recursion}
r_{m+2}(\mu) = b \ r_{m+1}(\mu) + \frac{c}{\gamma_0} \ \sum_{k=2}^{m} r_k(\mu) \ r_{m+2-k}(\mu),
\end{equation}
with the initial conditions
\[
r_1(\mu) = \beta_0, \qquad r_2(\mu) = \gamma_0.
\]
Here
\begin{equation}
\label{Restrictions}
\gamma_0 > 0, \quad \widetilde{\gamma}_0 \geq 0, \quad \widetilde{c} + \gamma_0 \geq 0, \quad c + \gamma_0 \geq 0.
\end{equation}
Condition~\eqref{Two-state-recursion} is equivalent to
\[
\frac{R^{\widetilde{\mu}, \mu}(z) - \widetilde{\beta}_0 z}{z^2} = \widetilde{\gamma}_0 + \widetilde{b} \ \frac{R^{\widetilde{\mu}, \mu}(z) - \widetilde{\beta}_0 z}{z} + \frac{\widetilde{c}}{\gamma_0} \ \frac{R^{\mu}(z) - \beta_0 z}{z} \ \frac{R^{\widetilde{\mu}, \mu}(z) - \widetilde{\beta}_0 z}{z} .
\]
\end{Prop}

\begin{proof}
Suppose $(\widetilde{\mu}, \mu)$ is a two-state free Meixner distribution. In particular, $\mu$ is not a point mass, so $\gamma_0 \neq 0$, and conditions~\eqref{Restrictions} are satisfied. Since $\mu$ is a free Meixner distribution, condition~\eqref{One-state-recursion} holds, see Remark~\ref{Remark:Free-Meixner}(d). The proof of \eqref{Two-state-recursion} is based on formula~\eqref{C-free-cumulant-expansion}. The initial conditions follow directly from that formula. Also,
\begin{align*}
R_3(\widetilde{\mu}, \mu) & = \widetilde{\gamma}_0 \widetilde{b} = \widetilde{b} \ R_2(\widetilde{\mu}, \mu), \\
R_4(\widetilde{\mu}, \mu) & = \widetilde{\gamma}_0 \widetilde{b}^2 + \widetilde{\gamma}_0 \widetilde{c} = \widetilde{b} \ R_3(\widetilde{\mu}, \mu) + \widetilde{c} \ \frac{r_2(\mu)}{\gamma_0} R_2(\widetilde{\mu}, \mu).
\end{align*}
For $m \geq 2$, if the element $\set{1}$ is an (outer) singleton of $\sigma$, it contributes $\widetilde{b}$, and the remaining classes of $\sigma$ form a partition in $\NC_{1,2}$ of $m-1$ elements. So the sum over all such partitions is $\widetilde{b} R_{m+1}(\widetilde{\mu}, \mu)$. If the element $1$ is not a singleton of $\sigma$, it belongs to an (outer) class $\set{1, k}$, for $2 \leq k \leq m$. This class contributes $\widetilde{c}$. The partition $\sigma$ restricted to the subset $\set{k+1, \ldots, m}$ is in $\NC_{1,2}$. $\sigma$ restricted to $\set{2, 3, \ldots, k-1}$ also is in $\NC_{1,2}$, but all the classes in this restriction are inner in $\sigma$, so the corresponding products contain only $b, c$ terms. Using equation~\eqref{free-cumulant-product}, it follows that the sum over all such partitions is exactly
\[
\widetilde{\gamma}_0 \widetilde{c} \sum_{k=2}^{m} \frac{r_{k}(\mu)}{\gamma_0} \frac{R_{m+2-k}(\widetilde{\mu}, \mu)}{\widetilde{\gamma}_0}
= \frac{\widetilde{c}}{\gamma_0} \sum_{k=2}^{m} r_{k}(\mu) R_{m+2-k}(\widetilde{\mu}, \mu).
\]
The formula for the generating functions follows.

Conversely, given a choice of parameters $b, c, \beta_0, \gamma_0, \widetilde{b}, \widetilde{c}, \widetilde{\beta}_0, \widetilde{\gamma}_0$ satisfying condition~\eqref{Restrictions}, the recursions determine the measures $(\widetilde{\mu}, \mu)$ uniquely, and by the first part of the argument, these are precisely two-state free Meixner distributions with this choice of parameters.
\end{proof}

\subsection{Laha-Lukacs characterization.}
A classical paper \cite{Laha-Lukacs} characterizes Meixner distributions in terms of certain conditional expectations. In \cite{Boz-Bryc}, the authors obtained a similar characterization of free Meixner distributions. The following is their result for the two-state free independence. Recall that $\mf{X}, \mf{Y}$ are $(\phi|\psi)$-free if all their mixed two-state free cumulants are zero, see the paper quoted below for the terminology. Note also that if $\phi$ is a tracial state, equation~\eqref{Conditioning} implies that the conditional expectation
\[
\phi[(\mf{X} - \mf{Y})^2 | \mf{S}] = C (4 \mf{I} + 2 b \mf{S} + c \mf{S}^2),
\]
is a quadratic function of $\mf{S}$, see Remark~\ref{Remark:Free-Meixner}(c). In two-state free probability theory, typically $\phi$ will not be tracial, and the conditional expectation with respect to it will not exist.

\begin{Thm*}[Theorem 2.1 in \cite{Boz-Bryc-Two-states}]
Let $(\mc{A}, \phi, \psi)$ be a two-state algebraic non-commutative probability space, and $\mf{X}, \mf{Y}$ two self-adjoint elements in it which are $(\phi|\psi)$-free and have the same distributions (with respect to both $\phi$ and $\psi$). Furthermore, assume that $\phi[\mf{X}] = 0$, $\phi[\mf{X}^2] = 1$. Let $\mf{S} = \mf{X} + \mf{Y}$ and suppose that there are $b, C \in \mf{R}$ and $c > -2$ such that
\begin{equation}
\label{Conditioning}
\phi[(\mf{X} - \mf{Y})^2 \mf{S}^n] = C \phi[(4 \mf{I} + 2 b \mf{S} + c \mf{S}^2) \mf{S}^n], n = 0, 1, 2 \ldots.
\end{equation}
Denote by $(\widetilde{\mu}_{\mf{S}}, \mu_{\mf{S}})$ the distribution of $\mf{S}$ in $(\mc{A}, \phi, \psi)$, and by $(\widetilde{\mu}, \mu)$ the corresponding distribution of $\mf{X}$ (and of $\mf{Y}$). Then
\begin{equation}
\label{Two-state-Laha-Lukacs}
1 + M^{\widetilde{\mu}_\mf{S}}(z) = \frac{2 + c - (2 b z + c) (1 + M^{\mu_\mf{S}}(z))}{2 + c - (4 z^2 + 2 b z + c) (1 + M^{\mu_\mf{S}}(z))}.
\end{equation}
\end{Thm*}

Bo{\.z}ejko and Bryc described the corresponding distributions more explicitly in particular cases corresponding to the Gaussian and Poisson regressions (that is, if $c = 0$). We now provide a complete description of the possible $\phi$ and $\psi$ distributions of $\mf{X}, \mf{Y}, \mf{S}$ which satisfy such regression relations. Recall that if $\phi = \psi$, then the $\psi$-distributions of $\mf{X}, \mf{Y}, \mf{S}$ are free Meixner distributions.

\begin{Prop}
In the context of the preceding theorem,
\[
J(\widetilde{\mu}_{\mf{S}}) =
\begin{pmatrix}
0, & b + (1 + c/2) \beta_0(\mu_{\mf{S}}), & \beta_1(\mu_{\mf{S}}), & \ldots \\
2, & (1 + c/2) \gamma_0(\mu_{\mf{S}}), & \gamma_1(\mu_{\mf{S}}), & \ldots
\end{pmatrix}.
\]
Suppose in addition that $\mu$ is a free Meixner distribution with Jacobi parameters \eqref{Free-Meixner-Jacobi}, with $\gamma_0 \geq 0$ and $c + \gamma_0 \geq 0$. Then provided that $c \geq -1$,
\[
J(\widetilde{\mu}) =
\begin{pmatrix}
0, & (b + c \beta_0) + \beta_0, & b + \beta_0, & b + \beta_0, & \ldots \\
1, & c \gamma_0 + \gamma_0, & c + \gamma_0, & c + \gamma_0, & \ldots
\end{pmatrix},
\]
Moreover, in this case $(\widetilde{\mu},\mu)$ is a two-state free Meixner distribution.
\end{Prop}

\begin{proof}
Using equations~\eqref{Relation:Eta-transform} and \eqref{Two-state-Laha-Lukacs},
\[
\begin{split}
\eta^{\widetilde{\mu}_{\mf{S}}}
& = 1 - (1 + M^{\widetilde{\mu}_{\mf{S}}})^{-1} \\
& = 1 - \frac{2 + c - (4 z^2 + 2 b z + c) (1 + M^{\mu_\mf{S}}(z))}{2 + c - (2 b z + c) (1 + M^{\mu_\mf{S}}(z))} \\
& = \frac{4 z^2 (1 + M^{\mu_\mf{S}}(z))}{2 + c - (2 b z + c) (1 + M^{\mu_\mf{S}}(z))} \\
& = - \frac{4 z^2}{2 b z + c - (2 + c) \left(1 + M^{\mu_{\mf{S}}}\right)^{-1}} \\
& = - \frac{4 z^2}{2 b z + c - (2 + c) (1 - \eta^{\mu_{\mf{S}}})} \\
& = - \frac{4 z^2}{-2 + 2 b z + (2 + c) \eta^{\mu_{\mf{S}}}} \\
& = \frac{2 z^2}{1 - b z - (1 + c/2) \eta^{\mu_{\mf{S}}}}.
\end{split}
\]
Comparing with the continued fraction expansion~\eqref{Continued-fraction-Eta}, we see that in terms of the Jacobi parameters of $\mu_{\mf{S}}$,
\[
J(\widetilde{\mu}_{\mf{S}}) =
\begin{pmatrix}
0, & b + (1 + c/2) \beta_0(\mu_{\mf{S}}), & \beta_1(\mu_{\mf{S}}), & \ldots \\
2, & (1 + c/2) \gamma_0(\mu_{\mf{S}}), & \gamma_1(\mu_{\mf{S}}), & \ldots
\end{pmatrix}.
\]
Now suppose that $\mf{X}$ has, with respect to $\psi$, a free Meixner distribution with Jacobi parameters \eqref{Free-Meixner-Jacobi}, which in particular means $\gamma_0 \geq 0$ and $c + \gamma_0 \geq 0$. Then by Theorem~\ref{Thm:Linear-one-state},
\[
J(\mu_{\mf{S}}) = J(\mu^{\boxplus 2}) =
\begin{pmatrix}
2\beta_0, & b + 2\beta_0, & b + 2\beta_0, & \ldots \\
2\gamma_0, & c + 2\gamma_0, & c + 2\gamma_0, & \ldots
\end{pmatrix}.
\]
So
\[
J(\widetilde{\mu}_{\mf{S}}) =
\begin{pmatrix}
0, & (b + c \beta_0) + 2 \beta_0, & b + 2\beta_0, & b + 2\beta_0, & \ldots \\
2, & c \gamma_0 + 2 \gamma_0, & c + 2\gamma_0, & c + 2\gamma_0, & \ldots
\end{pmatrix}.
\]
This is the $t=2$ case of
\[
J(\widetilde{\mu}_t) =
\begin{pmatrix}
0, & (b + c \beta_0) + \beta_0 t, & b + \beta_0 t, & b + \beta_0 t, & \ldots \\
t, & c \gamma_0 + \gamma_0 t, & c + \gamma_0 t, & c + \gamma_0 t, & \ldots
\end{pmatrix},
\]
with $\widetilde{b} = b + c \beta_0$ and $\widetilde{c} = c \gamma_0$. By setting $t=1$ instead, we get
\[
J(\widetilde{\mu}) =
\begin{pmatrix}
0, & (b + c \beta_0) + \beta_0, & b + \beta_0, & b + \beta_0, & \ldots \\
1, & c \gamma_0 + \gamma_0, & c + \gamma_0, & c + \gamma_0, & \ldots
\end{pmatrix}.
\]
This defines a positive measure provided that $c \geq -1$. Combining this with
\[
J(\mu) =
\begin{pmatrix}
\beta_0, & b + \beta_0, & b + \beta_0, & \ldots \\
\gamma_0, & c + \gamma_0, & c + \gamma_0, & \ldots
\end{pmatrix},
\]
we see that $(\widetilde{\mu},\mu)$ is a two-state free Meixner distribution.
\end{proof}

\subsection{Free quadratic harnesses}
In a series of papers starting with \cite{Bryc-Meixner}, Bryc and Weso{\l}owski (along with Matysiak and Szab{\l}owski) have investigated quadratic harnesses.
These are square-integrable processes $(X_t)_{t\ge0}$, with normalization $\mf{E}[X_t] = 0$, $\mf{E}[X_t X_s] = \min(t,s)$, such that $\mf{E}[X_t | \mc{F}_{s,u}]$ is a linear function of $X_s, X_u$ and $\Var[X_t |\mc{F}_{s,u}]$ is a quadratic function of $X_s, X_u$. Here $\mc{F}_{s,u}$ is the two-sided $\sigma$-field generated by $\set{X_r: r \in [0, s] \cup [u, \infty)}$. Then
\begin{equation}
\label{C-expectation}
\mf{E}[X_t | \mc{F}_{s,u}] = \frac{u - t}{u - s} X_s + \frac{t - s}{u - s} X_u
\end{equation}
and under certain technical assumptions (see \cite{Bryc-Matysiak-Wesolowski-q-commutation}),
\begin{equation}
\label{C-variance}
\begin{split}
\Var[X_t |\mc{F}_{s,u}] & = \frac{(u - t) (t - s)}{u (1 + \sigma s) + \tau - \gamma s} \biggl( 1 + \sigma \frac{(u X_s - s X_u)^2}{(u - s)^2} + \tau \frac{(X_u - X_s)^2}{(u - s)^2} \\
&\quad + \eta \frac{u X_s - s X_u}{u - s} + \theta \frac{X_u - X_s}{u - s}  - (1 - \gamma) \frac{(X_u - X_s)(u X_s - s X_u)}{(u - s)^2} \biggr).
\end{split}
\end{equation}
The authors proved the existence of such processes for a large range of parameters $\sigma, \tau, \eta, \theta, \gamma$, in particular connecting the analysis to the Askey-Wilson measures in \cite{Bryc-Wesolowski-Askey-Wilson} (the standard Askey-Wilson parameter is $q = \gamma + \sigma \tau$). One reason for the interest in this analysis comes from numerous particular cases.
\begin{enumerate}
\item
If $\gamma = 1$ and $\sigma = \eta = 0$, the processes automatically have classically independent increments, and each $X_t$ has a Meixner distribution, see \cite{Wes-commutative}.
\item
For $\gamma = \sigma = \eta = 0$, the processes are classical versions of processes with free independent increments, and have free Meixner distributions.
\item
For $\sigma = \eta = 0$ and $-1 \leq \gamma = q < 1$, the corresponding orthogonal martingale polynomials have Jacobi parameters
\[
\left(
\begin{array}{rl}
\beta_n(t) & = \theta [n]_q \\
\gamma_n(t) & = [n+1]_q (t + \tau [n]_q) \\
\end{array}
\right),
\]
where $[n]_q:=1+q+\ldots+q^{n-1}$ is the $q$-integer. If $\tau = 0$, the process is a (classical version of a) $q$-Poisson process from \cite{AnsQCum}. The case where in addition, $\theta = 0$ was considered even earlier in \cite{Bryc-Stationary-random-fields} and corresponds to the $q$-Brownian motion \cite{BKSQGauss}. The challenge of interpreting the general processes with $\sigma = \eta = 0$ as ``processes with $q$-independent increments'' remains open.
\item
Finally, for $\gamma = \sigma = \tau = 0$, the free bi-Poisson processes from \cite{Bryc-Wesolowski-Bi-Poisson} are shown, in Section~4 of that paper, to have increments freely independent with respect to a pair of states.
\end{enumerate}
We will now extend the last result above. Proposition~4.3 of \cite{Bryc-Matysiak-Wesolowski-q-commutation} states that for
\[
q = \gamma + \sigma \tau = 0,
\]
there exist orthogonal martingale polynomials for the process. They satisfy recursion relations $P_0(x,t) = 1$,
\begin{align*}
x P_0(x,t) & = P_1(x,t), \\
x P_1(x,t) & = (1 + \sigma t) P_2(x,t) + (u t + v) P_1(x,t) + t P_0(x,t), \\
x P_2(x,t) & = (1 + \sigma t) P_3(x,t) + \left( \frac{u + \sigma v}{1 - \sigma \tau} t + \frac{v + \tau u}{1 - \sigma \tau} \right) P_2(x,t) + \frac{1 + u v}{1 - \sigma \tau} (t + \tau) P_1(x,t), \\
x P_n(x,t) & = (1 + \sigma t) P_{n+1}(x,t) + \left( \frac{u + \sigma v}{1 - \sigma \tau} t + \frac{v + \tau u}{1 - \sigma \tau} \right) P_n(x,t) + \frac{1 + u v}{(1 - \sigma \tau)^2} (t + \tau) P_{n-1}(x,t) \\
\end{align*}
for $n \geq 3$, where
\[
u = \frac{\eta + \sigma \theta}{1 - \sigma \tau}, \quad v = \frac{\tau \eta + \theta}{1 - \sigma \tau}
\]
and as long as
\begin{equation*}
\label{B-W-normalization}
1 + u v > 0, 0 \leq \sigma \tau < 1.
\end{equation*}
Note that the coefficients in this recursion are linear in $t$, so the corresponding tridiagonal matrix is of the form in equation~\eqref{Jacobi-matrix}. Moreover the coefficients are constant for $n \geq 2$, so this class is very close to the families considered in this paper. However, the corresponding polynomials are not monic. The Jacobi parameters for the monic orthogonal polynomials for this process (which are not martingale polynomials) are quadratic in $t$ (see Corollary~\ref{Cor:Jacobi}). Therefore they do not form a semigroup with respect to any of the convolutions considered in this paper, unless $\sigma = 0$.

\begin{Remark}
It is a fundamental observation of Bryc et al. that if a process $(X_t)$ satisfies properties \eqref{C-expectation} and \eqref{C-variance}, so does the process $Y_t = t X_{1/t}$, as long as parameters $\sigma \leftrightarrow \tau$ and $\eta \leftrightarrow \theta$ are interchanged. In particular, the class of free quadratic harnesses is closed under such a time-reversal operation. On the other hand, it is easy to see that the class of families $\set{\widetilde{\mu}_t}$ which may arise in Theorem~\ref{Thm:Linear-two-state} is \emph{not} closed under this operation. In fact, the largest sub-family of this class which is closed in this way corresponds to $\tilde{b} = b$, $\tilde{c} = c = 0$, which is precisely (up to re-scaling) the class of free bi-Poisson processes.
\end{Remark}

\begin{Prop}
Let $q = 0$, $\sigma = 0$, $\tau > 0$, $\eta \neq 0$, and $1 + \eta(\tau \eta + \theta) > 0$. Denoting by $\widetilde{\mu}_t$ the distribution of $X_t$, there exist $\set{\mu_t}$ such that the pairs $\set{(\widetilde{\mu}_t, \mu_t)}$ form a two-state free convolution semigroup. Also in this case, $\widetilde{\rho}$ is a free Poisson distribution.
\end{Prop}

\begin{proof}
Since $\sigma = 0$, we have $u = \eta$, $v = \tau \eta + \theta$. Thus the recursion above gives
\[
J(\widetilde{\mu}_t) =
\begin{pmatrix}
0, & \tau \eta + \theta + \eta t, & 2 \tau \eta + \theta + \eta t, & 2 \tau \eta + \theta + \eta t, & \ldots \\
t, & (1 + \eta (\tau \eta + \theta)) (t + \tau), & (1 + \eta (\tau \eta + \theta)) (t + \tau), & (1 + \eta (\tau \eta + \theta)) (t + \tau), & \ldots
\end{pmatrix},
\]
Identifying the coefficients in the recursions above with parameters in our Theorem~\ref{Thm:Linear-two-state} gives
\[
\widetilde{\beta}_0 = 0, \quad \beta_0 = \eta, \quad \widetilde{b} = \tau \eta + \theta, \quad b = 2 \tau \eta + \theta
\]
and
\[
\widetilde{\gamma}_0 = 1, \quad \gamma_0 = 1 + \eta (\tau \eta + \theta), \quad \widetilde{c} = c = \tau (1 + \eta (\tau \eta + \theta)).
\]
Thus taking $\mu_t$ to be the free Meixner distributions with
\[
J(\mu_t) =
\begin{pmatrix}
\eta t, & 2 \tau \eta + \theta + \eta t, & 2 \tau \eta + \theta + \eta t, & \ldots \\
(1 + \eta (\tau \eta + \theta)) t, & (1 + \eta (\tau \eta + \theta)) (t + \tau), & (1 + \eta (\tau \eta + \theta)) (t + \tau), & \ldots
\end{pmatrix},
\]
the pairs $\set{(\widetilde{\mu}_t, \mu_t)}$ form a two-state free convolution semigroup.  Also,
\[
J(\widetilde{\rho}) =
\begin{pmatrix}
\tau \eta + \theta, & 2 \tau \eta + \theta, & 2 \tau \eta + \theta, & \ldots \\
\tau (1 + \eta (\tau \eta + \theta)), & \tau (1 + \eta (\tau \eta + \theta)), & \tau (1 + \eta (\tau \eta + \theta)), & \ldots
\end{pmatrix}
\]
and $\widetilde{\rho}$ is a free Poisson distribution.
\end{proof}

\begin{Remark}
Setting $\tau = 0$, $\theta \neq 0$ corresponds to the free bi-Poisson process with
\[
J(\mu_t) =
\begin{pmatrix}
\eta t, & \theta + \eta t, & \theta + \eta t, & \ldots \\
(1 + \eta \theta) t, & (1 + \eta \theta) t, & (1 + \eta \theta) t, & \ldots
\end{pmatrix},
\]
a free Poisson distribution, $\widetilde{\rho} = \delta_\theta$, and
\[
J(\widetilde{\mu}_t) =
\begin{pmatrix}
0, & \theta + \eta t, & \theta + \eta t, & \theta + \eta t, & \ldots \\
t, & (1 + \eta \theta) t, & (1 + \eta \theta) t, & (1 + \eta \theta) t, & \ldots
\end{pmatrix},
\]
a free Meixner distribution. Additionally setting $\theta = 0$ gives $\mu_t$ a (non-centered) semicircular distribution and $\widetilde{\mu}_t = \mu_{\eta, 0}^{\boxplus t}$ a free Poisson distribution.

On the other hand, restriction to $\eta = 0$ gives
\[
\mu_t = \widetilde{\mu}_t = \mu_{\theta, \tau}^{\boxplus t}
\]
equal free Meixner distributions, and $\widetilde{\rho}$ a semicircular distribution.
\end{Remark}


\def\cprime{$'$}
\providecommand{\bysame}{\leavevmode\hbox to3em{\hrulefill}\thinspace}
\providecommand{\MR}{\relax\ifhmode\unskip\space\fi MR }
\providecommand{\MRhref}[2]{%
  \href{http://www.ams.org/mathscinet-getitem?mr=#1}{#2}
}
\providecommand{\href}[2]{#2}

\end{document}